\documentclass[11pt,a4paper,reqno]{amsart} \usepackage{amsmath,amssymb,tikz,color}
\usetikzlibrary{positioning}
\usepackage{fullpage}

\usepackage[vcentermath]{youngtab}
\usepackage{youngtab}
\usetikzlibrary{chains}
\usepackage{subfigure}

\usepackage[utf8]{inputenc}

\newcommand{\DecRib}{\mathsf{LRib}}
\newcommand{\DR}{\DecRib}

\newcommand{\mynewpage}{}

\newcommand{\Para}{\mathsf{Para}}
\newcommand{\Rib}{\mathsf{Rib}}
\newcommand{\MRib}{\mathsf{LRib}}
\newcommand{\Ribbon}{\Rib}
\newcommand{\MPara}{\mathsf{LPara}}
\newcommand{\DecPara}{\mathsf{LPara}}

\newcommand{\Rec}{\mathsf{Rec}}
\newcommand{\Recmin}{\mathsf{Rec}^{\mathrm{min}}}
\newcommand{\Recmindec}{\mathsf{Rec}^{\mathrm{min,dec}}}
\newcommand{\Recdec}{\mathsf{Rec}^{\mathrm{dec}}}

\newcommand{\EW}{{\mathsf{EW}}}
\newcommand{\MEW}{{\mathsf{MEW}}}
\newcommand{\DEW}{\mathsf{DEW}}

\newcommand{\mycomp}[1]{\overline{#1}}
\newcommand{\rightangle}{cornersupport}
\newcommand{\bRmMr}{f_{\mathsf{RmMr}}}

\newcommand{\bEMr}{f_{\mathsf{EMr}}}

\newcommand{\bERm}{f_{\mathsf{ERm}}}

\newcommand{\mydecrib}{\phi}
\newcommand{\bounce}{\mathrm{bounce}}
\newcommand{\GETOUT}[1]{}
\newcommand{\magic}{h}

\DeclareMathAlphabet{\mathpzc}{OT1}{pzc}{m}{it}
\def\lgrey{0.9}
\def\softgrey{0.8}
\definecolor{lightgrey}{rgb}{\lgrey,\lgrey,\lgrey}
\definecolor{grey}{rgb}{\softgrey,\softgrey,\softgrey}

 \newtheorem{theorem}{Theorem}[section]

 \theoremstyle{definition}
 \newtheorem{example}[theorem]{Example}
 \newtheorem{proposition}[theorem]{Proposition}

 \newtheorem{definition}[theorem]{Definition}
 
 \newtheorem{remark}[theorem]{Remark}

 \theoremstyle{remark}

\definecolor{mygreen}{RGB}{80,160,80}

\newcommand\mps[1]{\marginpar[\raggedleft\tiny\sf\textcolor{mygreen}{#1}]{\tiny\sf\textcolor{blue}{#1}}}
\renewcommand\mps[1]{}

\newcommand{\EWdiagramlab}[4]{
  \begin{tikzpicture}
\def\zs{0^{\star}}
\def\os{1^{\star}}
\def\step{0.43}
\draw (0,0) node [anchor = north west]  {\young(#1)};
\foreach[count=\y] \lab in {#2}
	\draw (0.25,-\y*\step+0.05) node [anchor=east]{\tiny{$v_{\lab}$}};
\foreach[count=\x] \lab in {#3}
	\draw (\x*\step-0.05,0.2) node [anchor=north]{\tiny{$v_{\lab}$}};
\draw [left] (-.25,-1) node [anchor = east]{#4};
\end{tikzpicture}
}
\newcommand{\EWdiagramlabwithin}[4]{
\def\zs{0^{\star}}
\def\os{1^{\star}}
\def\step{0.43}
\draw (0,0) node [anchor = north west]  {\young(#1)};
\foreach[count=\y] \lab in {#2}
	\draw (0.25,-\y*\step+0.05) node [anchor=east]{\tiny{$v_{\lab}$}};
\foreach[count=\x] \lab in {#3}
	\draw (\x*\step-0.05,0.2) node [anchor=north]{\tiny{$v_{\lab}$}};
\draw [left] (-.25,-1) node [anchor = east]{#4};
}

\title[]{Parallelogram polyominoes and rectangular EW-tableaux: correspondences through the sandpile model}
\author[A. Alofi]{Amal Alofi}
\address{UCD School of Mathematics and Statistics, University College Dublin, Dublin 4, Ireland} \email{amal.alofi@ucdconnect.ie}
\author[M. Dukes]{Mark Dukes}
\address{UCD School of Mathematics and Statistics, University College Dublin, Dublin 4, Ireland} \email{mark.dukes@ccc.oxon.org}

\begin{document}

\begin{abstract}
This paper establishes connections between EW-tableaux and parallelogram polyominoes by using recent research regarding the sandpile model on the complete bipartite graph.
This paper presents and proves a direct bijection between rectangular EW-tableaux and labelled ribbon parallelogram polyominoes. 
The significance of this is that allows one to move between these objects without the need for `recurrent configurations', the central object which previously tied this work together.
It introduces the notion of a marked rectangular EW-tableaux that exactly encode all recurrent configurations of the sandpile model on the complete bipartite graph. 
This 
shows how non-cornersupport entries that featured in previous work can be utilized 
in a simple but important way in relation to EW-tableaux.
It lifts the bijection between rectangular EW-tableaux and labelled ribbon parallelogram polyominoes 
to a bijection between marked rectangular EW-tableaux and labelled parallelogram polyominoes. 
This bijection helps us to fully understand the aspects of these very different objects that are, in a sense, different sides of the same coin.
\end{abstract}

\maketitle

\section{Introduction}\label{sec:intro}
In a series of recent papers, several classes of combinatorial objects have been shown to be in one-to-one correspondence with recurrent configurations of the Abelian sandpile model (ASM) on different classes of graphs~\cite{aadhlb,aadlb,dlb,sss,dsss,dasm}.
These connections have shown a certain combinatorial richness emerges from studying the behaviour of Dhar's burning algorithm -- an algorithm that checks a configuration for recurrence -- on different graph classes.
In this paper we will study these correspondences with a slightly different goal in mind.
To explain our motivation it is necessary to recall several results from these papers that are diagrammatically summarised in Figure~\ref{fig:summary}.

A polyomino is a planar generalization of a domino that consists of unit cells having integral coordinates and is connected is some manner.
Dukes and Le Borgne~\cite{dlb} established a bijection from the set of all {\it{weakly decreasing}}
recurrent configurations of the ASM on the complete bipartite graph $K_{m,n}$, $\Recdec(K_{m,n};v_0)$,
to the set of all {\it{parallelogram}} polyominoes having a $m\times n$ bounding box, $\Para_{m,n}$.
They also showed how that correspondence could be lifted to one between all recurrent configurations, $\Rec(K_{m,n};v_0)$, and 
a set of labelled parallelogram polyominoes. 
Minimal recurrent configurations, $\Recmin(K_{m,n};v_0)$, were shown to correspond 
to those labelled parallelogram polyominoes that have minimal area, $\MRib_{m,n}$.

In Dukes et al.~\cite{dsss}, the authors gave a bijection between all recurrent configurations of the ASM on a Ferrers graph $G(F)$
and the set of all `decorated' EW-tableaux having shape $F$. 
EW-tableaux are certain 0/1 fillings of a Ferrers diagram that were first defined by Ehrenborg and van Willigenburg~\cite{evw}.
If the Ferrers diagram $F$ is a rectangular diagram, then the corresponding graph $G(F)$ is a complete bipartite graph.
It is in this setting that the results of \cite{dlb} can be compared to those of \cite{dsss} and we summarize these in Figure~\ref{fig:summary}.

The correspondences in Figure~\ref{fig:summary} have a central spine consisting of different types of recurrent configurations.
Dashed lines indicate containment within, or lifting to, a structure above.
On the left hand side of this figure the correspondences of the papers ~\cite{sss,dsss} is summarised.
On the right hand side the correspondences with different types of parallelogram polyominoes are illustrated.

The composition of bijections allows one, in theory, to map from a (type of) 0/1 filling of a rectangular tableaux to a (type of) parallelogram polyomino.
Our motivation is drawn from the following question:
is it possible to describe the composition of the two bijections without using recurrent configurations of the ASM?
As parallelogram polyominoes can be considered to be a type of 0/1 tableaux, wherein 1s represent filled cells and 0s unfilled cells, 
this paper studies the equivalence of {\it{marked EW-rectangular tableaux}} and {\it{labelled parallelogram polyominoes}}.

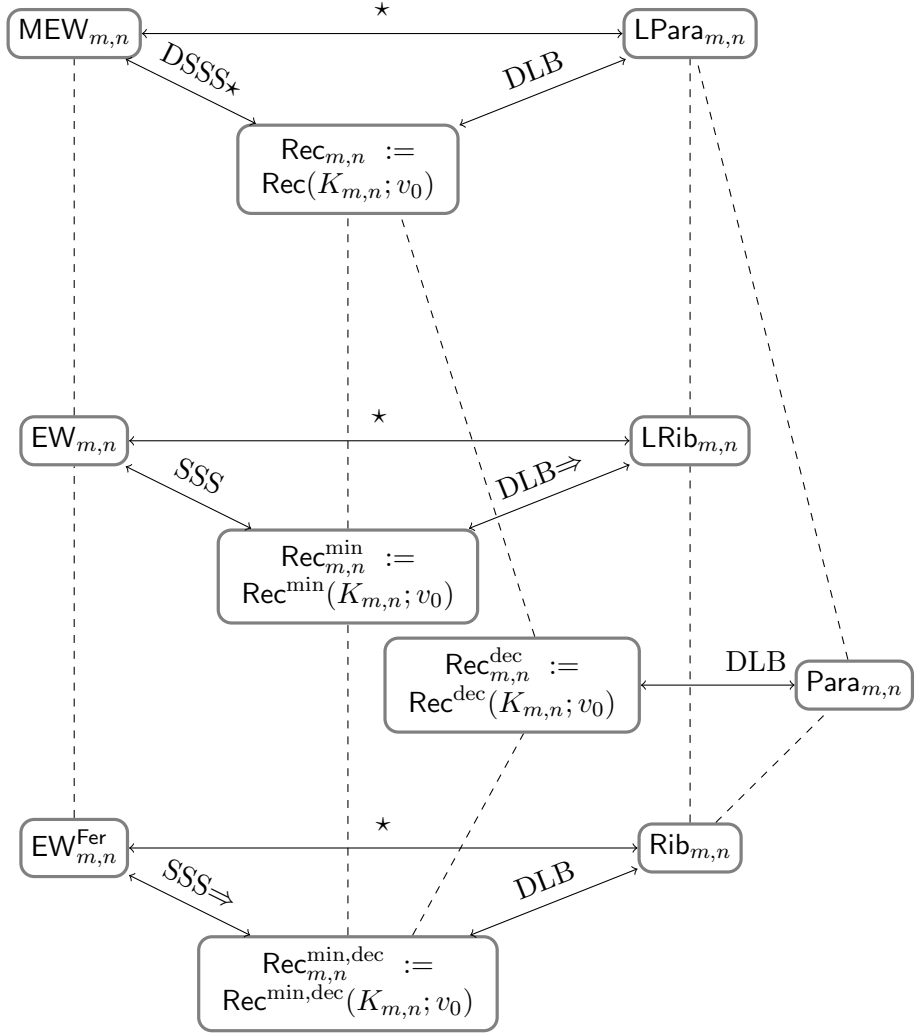
\begin{figure}[!h]
\begin{tikzpicture}[scale=0.45, 
	every node/.style={rectangle,rounded corners=5pt,minimum size=6mm,very thick,draw=black!50,fill=white} ]
\def\xs{120}\def\ys{100}
\def\x{4}\def\y{4}
\node (EWFER) at (-2*\x,1*\y) {$\EW^{\textsf{Fer}}_{m,n}$};
\node (RMD) at (0,0) {$\begin{array}{c} \Recmindec_{m,n}~:=\\ \Recmindec(K_{m,n};v_0)\end{array}$};
\node (RD) at (1.2*\x,2.2*\y) {$\begin{array}{c} \Recdec_{m,n}~:=\\ \Recdec(K_{m,n};v_0)\end{array}$};
\node (RM) at (0*\x,3*\y) {$\begin{array}{c} \Recmin_{m,n}~:=\\ \Recmin(K_{m,n};v_0)\end{array}$};
\node (EW) at (-2*\x,4*\y) {$\EW_{m,n}$};
\node (MEW) at (-2*\x,7*\y) {$\MEW_{m,n}$};
\node (RIB) at (2.5*\x,1*\y) {$\Rib_{m,n}$};
\node (MRIB) at (2.5*\x,4*\y) {$\MRib_{m,n}$};
\node (PARA) at (3.7*\x,2.2*\y) {$\Para_{m,n}$};
\node (R) at (0*\x,6*\y) {$\begin{array}{c} \Rec_{m,n}~:=\\ \Rec(K_{m,n};v_0)\end{array}$};
\node (MPARA) at (2.5*\x,7*\y) {$\MPara_{m,n}$};
\draw[dashed] (EWFER) -- (EW) -- (MEW);
\draw[dashed] (RMD) -- (RM) -- (R);
\draw[transform canvas={xshift=0.5cm},dashed] (RMD) -- (RD) -- (R);
\draw[dashed] (RIB) -- (MRIB) -- (MPARA);
\draw[dashed] (RIB) -- (PARA) -- (MPARA);
\draw[<->] (MEW) -- (MPARA) node[draw=none,fill=none,font=\normalsize,midway,above] {$\star$}; 
\draw[<->] (MEW) -- (R) node[draw=none,fill=none,font=\normalsize,midway,above,sloped] {DSSS$\star$}; 
\draw[<->] (R) -- (MPARA) node[draw=none,fill=none,font=\normalsize,midway,above,sloped] {DLB}; 
\draw[<->] (EW) -- (MRIB) node[draw=none,fill=none,font=\normalsize,midway,above] {$\star$};; 
\draw[<->] (EW) -- (RM) node[draw=none,fill=none,font=\normalsize,midway,above,sloped] {SSS};
\draw[<->] (RM) -- (MRIB) node[draw=none,fill=none,font=\normalsize,midway,above,sloped] {DLB$\Rightarrow$}; 
\draw[<->] (EWFER) -- (RIB) node[draw=none,fill=none,font=\normalsize,midway,above] {$\star$};; 
\draw[<->] (EWFER) -- (RMD) node[draw=none,fill=none,font=\normalsize,midway,above,sloped] {SSS$\Rightarrow$};
\draw[<->] (RMD) -- (RIB) node[draw=none,fill=none,font=\normalsize,midway,above,sloped] {DLB}; 
\draw[<->] (RD) -- (PARA) node[draw=none,fill=none,font=\normalsize,near end,above,sloped] {DLB}; 
\end{tikzpicture}
\caption{An overview of the relations between the sets studied in this paper. 
Arrows on both ends of a solid line indicate a bijection between the two objects. 
Known bijections are indicated with letters. Those with a $\star$ are the topic of this paper. 
Dashed lines indicate subset containment, or a lifting of a structure to a more general object, from bottom to top.
The attributions are DLB~\cite{dlb}, SSS~\cite{sss}, DSSS~\cite{dsss}.
The symbol $\Rightarrow$ indicates a bijection is implied by results in a paper.
}
\label{fig:summary}
\end{figure}

This paper is organised as follows.
In Section 2 we will give an overview of the combinatorial objects used in this paper.
In Section 3 we present and prove a bijection between rectangular EW-tableaux and the set of labelled ribbon parallelogram polyominoes.
In Section 4 we lift the bijection presented in Section 3 to a more general level by introducing the notion of marked EW-tableaux.
A direct bijection between this set and the set of all labelled parallelogram polyominoes is proven.

\section{Two combinatorial objects}
\label{sec:defs}
In this section we will recall some of the notation, concepts, and results from the papers mentioned in Section~\ref{sec:intro} that will be necessary for our work.

\subsection{Parallelogram polyominoes and their decorations}
Parallelogram polyominoes, also known as staircase polyominoes, are polyominoes that are contained within two staircase shapes that only touch at their endpoints.
More formally: a parallelogram polyomino with a $m\times n$ bounding box is defined by two lattice paths from (0,0) to $(m,n)$ that
take unit north and east steps and do not touch other than at the origin and $(m,n)$.
A parallelogram polyomino is called a {\it{ribbon parallelogram polyomino}} if it has minimal area, in the case of a $m\times n$ the minimal area is $m+n-1$.

\begin{example}\label{firstpp}
The following shaded shape is parallelogram polyomino with $(m,n)=(7,3)$:\\
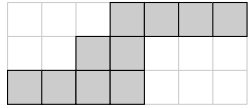
    \begin{center}
    \def\oce{0.3}
        \def\mskip{1}
    \begin{tikzpicture}[scale=0.45]
    \newcommand{\mybox}[2]{\draw [fill=grey] ({#1},{#2}) rectangle (1+#1,1+#2);}
    \newcommand{\bleep}[4]{\draw [line width=1.5pt] (\mskip*#1,\mskip*#2) -- (\mskip*#3,\mskip*#4);}
    \newcommand{\vtxst}[3]{\draw (0.25+\mskip*{#1},0.5+\mskip*{#2}) node{#3};}
    \newcommand{\htxst}[3]{\draw (0.5+\mskip*{#1},0.3+\mskip*{#2}) node{#3};}
        \foreach \x in {0,1,2,...,7}
                \draw [color=grey] (0+\mskip*\x,0) -- (0+\mskip*\x,3*\mskip);
        \foreach \y in {0,1,2,...,3}
                \draw [color=grey] (0,\mskip*\y) -- (7*\mskip,\mskip*\y);
    \mybox{0}{0} \mybox{1}{0} \mybox{2}{0} \mybox{3}{0} \mybox{2}{1}
    \mybox{3}{1} \mybox{3}{2} \mybox{4}{2} \mybox{5}{2} \mybox{6}{2}
    \end{tikzpicture}
    \endpgfgraphicnamed
    \end{center}
The polyomino is not a ribbon parallelogram polyomino, but can be made into one by removing either the rightmost box on the lowest row, or removing the leftmost box in the middle row.
\end{example}
One aspect of parallelogram polyominoes has proven crucial in their analysis with regard to the sandpile model.
Given a $(m,n)-$parallelogram polyomino $P$, we define the {\it{bounce path}} of $P$ to be the path that starts at $(m,n)$, goes west to $(m-1,n)$, 
turns and goes south until encountering the lower path of $P$, turns and moves west until encountering the upper path of $P$, and then repeats these last two steps until it encounters the origin.

\begin{example}
The bounce path of the parallelogram polyomino from Example~\ref{firstpp} is illustrated in this diagram.\\
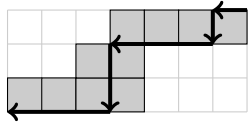
    \begin{center}
    \def\oce{0.3}
        \def\mskip{1}
    \begin{tikzpicture}[scale=0.45]
    \newcommand{\mybox}[2]{\draw [fill=grey] ({#1},{#2}) rectangle (1+#1,1+#2);}
    \newcommand{\bleep}[4]{\draw [line width=1.5pt] (\mskip*#1,\mskip*#2) -- (\mskip*#3,\mskip*#4);}
    \newcommand{\vtxst}[3]{\draw (0.25+\mskip*{#1},0.5+\mskip*{#2}) node{#3};}
    \newcommand{\htxst}[3]{\draw (0.5+\mskip*{#1},0.3+\mskip*{#2}) node{#3};}
        \foreach \x in {0,1,2,...,7}
                \draw [color=grey] (0+\mskip*\x,0) -- (0+\mskip*\x,3*\mskip);
        \foreach \y in {0,1,2,...,3}
                \draw [color=grey] (0,\mskip*\y) -- (7*\mskip,\mskip*\y);
    \mybox{0}{0} \mybox{1}{0} \mybox{2}{0} \mybox{3}{0} \mybox{2}{1}
    \mybox{3}{1} \mybox{3}{2} \mybox{4}{2} \mybox{5}{2} \mybox{6}{2}
    \draw [->,line width=1.5pt] (7*\mskip,3*\mskip) -- (6*\mskip,3*\mskip);
    \draw [->,line width=1.5pt] (6*\mskip,3*\mskip) -- (6*\mskip,2*\mskip);
    \draw [->,line width=1.5pt] (6*\mskip,2*\mskip) -- (3*\mskip,2*\mskip);
    \draw [->,line width=1.5pt] (3*\mskip,2*\mskip) -- (3*\mskip,0*\mskip);
    \draw [->,line width=1.5pt] (3*\mskip,0*\mskip) -- (0*\mskip,0*\mskip);
    \draw [line width=0,color=white] (-0.1*\mskip,-0.3*\mskip) rectangle (7.1*\mskip,3.3*\mskip);
    \end{tikzpicture}
    \end{center}
\end{example}

The bounce path was used in the original paper~\cite{dlb} linking parallelogram polyominoes and the sandpile model on $K_{m,n}$ to 
record the vertices that toppled during consecutive parallel topplings of a configuration when Dhar's burning algorithm was applied to it. 

We can label a parallelogram polyomino in such a way to create a collection of labelled objects that are extremely useful for our purposes.
Label the horizontal steps of the upper path of such a polyomino with unique entries from the set $\{v_0,\ldots,v_{m-1}\}$.
The rightmost such step must have label $v_0$. A sequence of steps having the same height must have labels that are decreasing from left to right. 
Label the vertical steps of the lower path with unique entries from the set $\{v_m,\ldots,v_{m+n-1}\}$. 
Vertical steps that are the same distance from the vertical axis must have labels that are increasing from top to bottom.

\begin{example} \ \\
    \begin{center}
    \def\oce{0.3}
        \def\mskip{1}
    \begin{tikzpicture}[scale=0.45]
    \newcommand{\mybox}[2]{\draw [fill=grey] ({#1},{#2}) rectangle (1+#1,1+#2);}
    \newcommand{\bleep}[4]{\draw [line width=1.5pt] (\mskip*#1,\mskip*#2) -- (\mskip*#3,\mskip*#4);}
    \newcommand{\vtxst}[3]{\draw (0.25+\mskip*{#1},0.5+\mskip*{#2}) node{#3};}
    \newcommand{\htxst}[3]{\draw (0.5+\mskip*{#1},0.3+\mskip*{#2}) node{#3};}
        \foreach \x in {0,1,2,...,7}
                \draw [color=grey] (0+\mskip*\x,0) -- (0+\mskip*\x,3*\mskip);
        \foreach \y in {0,1,2,...,3}
                \draw [color=grey] (0,\mskip*\y) -- (7*\mskip,\mskip*\y);
    \mybox{0}{0} \mybox{1}{0} \mybox{2}{0} \mybox{3}{0} \mybox{2}{1}
    \mybox{3}{1} \mybox{3}{2} \mybox{4}{2} \mybox{5}{2} \mybox{6}{2}
    \draw [dotted,<-] (\oce+7*\mskip,2.5*\mskip) -- (8*\mskip,2.5*\mskip);
    \node [right] () at (8*\mskip,2.5*\mskip) {$\{v_8\}$};
    \draw [dotted,<-] (\oce+4*\mskip,0.5*\mskip) -- (8*\mskip,0.5*\mskip);
    \node [right] () at (8*\mskip,0.5*\mskip) {$\{v_7,v_9\}$};
    \draw [dotted,<-] (5.5*\mskip,3*\mskip+\oce) -- (5.5*\mskip,4.1*\mskip);
    \node [above] () at (5.2*\mskip,4*\mskip) {$\{v_5,v_4,v_1,v_0\}$};
    \draw [dotted,<-] (2.5*\mskip,2*\mskip+\oce) -- (2.5*\mskip,5.1*\mskip);
    \node [above] () at (2.5*\mskip,5*\mskip) {$\{v_3\}$};
    \draw [dotted,<-] (0.5*\mskip,1*\mskip+\oce) -- (0.5*\mskip,4.1*\mskip);
    \node [above] () at (0.5*\mskip,4*\mskip) {$\{v_6,v_2\}$};
    \draw [line width=0,color=white] (-0.1*\mskip,-0.3*\mskip) rectangle (7.1*\mskip,3.3*\mskip);
    \end{tikzpicture}
    \end{center}
\end{example}

Conventions differ from paper to paper and in this paper we will choose a slightly different way to represent parallelogram polyominoes.
We will first illustrate this and then give the formal definition.

\begin{example}\label{dppe}
Replace filled cells with 1s and replace unfilled cells with 0s.
Label the top of the tableau with the vertex labels that correspond the horizontal steps of upper path in those columns.
The top right horizontal step by convention has label $v_0$.
For horizontal steps having the same height, ensure the labels are decreasing from left to right. 
For example, there are 3! ways to do this for the horizontal steps corresponding to $\{v_1,v_4,v_5\}$, but we label the columns from right to left that correspond to these steps on the bounce path with $v_1$, $v_4$, and $v_5$.
For vertical steps on the lower path, do the same.
Finally, rotate the original diagram anti-clockwise a quarter turn.
    \begin{center}
    \def\oce{0.3}
        \def\mskip{1}
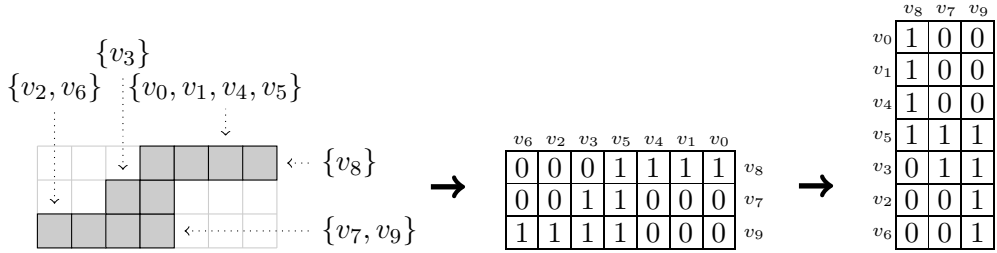
\begin{tabular}{lll}
    \begin{tikzpicture}[scale=0.45]
    \newcommand{\mybox}[2]{\draw [fill=grey] ({#1},{#2}) rectangle (1+#1,1+#2);}
    \newcommand{\bleep}[4]{\draw [line width=1.5pt] (\mskip*#1,\mskip*#2) -- (\mskip*#3,\mskip*#4);}
    \newcommand{\vtxst}[3]{\draw (0.25+\mskip*{#1},0.5+\mskip*{#2}) node{#3};}
    \newcommand{\htxst}[3]{\draw (0.5+\mskip*{#1},0.3+\mskip*{#2}) node{#3};}
        \foreach \x in {0,1,2,...,7}
                \draw [color=grey] (0+\mskip*\x,0) -- (0+\mskip*\x,3*\mskip);
        \foreach \y in {0,1,2,...,3}
                \draw [color=grey] (0,\mskip*\y) -- (7*\mskip,\mskip*\y);
    \mybox{0}{0} \mybox{1}{0} \mybox{2}{0} \mybox{3}{0} \mybox{2}{1}
    \mybox{3}{1} \mybox{3}{2} \mybox{4}{2} \mybox{5}{2} \mybox{6}{2}
    \draw [dotted,<-] (\oce+7*\mskip,2.5*\mskip) -- (8*\mskip,2.5*\mskip);
    \node [right] () at (8*\mskip,2.5*\mskip) {$\{v_8\}$};
    \draw [dotted,<-] (\oce+4*\mskip,0.5*\mskip) -- (8*\mskip,0.5*\mskip);
    \node [right] () at (8*\mskip,0.5*\mskip) {$\{v_7,v_9\}$};
    \draw [dotted,<-] (5.5*\mskip,3*\mskip+\oce) -- (5.5*\mskip,4.1*\mskip);
    \node [above] () at (5.2*\mskip,4*\mskip) {$\{v_0,v_1,v_4,v_5\}$};
    \draw [dotted,<-] (2.5*\mskip,2*\mskip+\oce) -- (2.5*\mskip,5.1*\mskip);
    \node [above] () at (2.5*\mskip,5*\mskip) {$\{v_3\}$};
    \draw [dotted,<-] (0.5*\mskip,1*\mskip+\oce) -- (0.5*\mskip,4.1*\mskip);
    \node [above] () at (0.5*\mskip,4*\mskip) {$\{v_2,v_6\}$};
    \draw [line width=0,color=white] (-0.1*\mskip,-0.3*\mskip) rectangle (7.1*\mskip,3.3*\mskip);
	\draw [->,line width=2pt] (11.5*\mskip,1.5*\mskip+\oce) -- (12.5*\mskip,1.5*\mskip+\oce);
    \end{tikzpicture} 
&
\begin{tikzpicture}
\def\step{0.43}
\draw (0,0) node [anchor = north west]  {\young(0001111,0011000,1111000)};
\foreach[count=\y] \lab in {8,7,9}
        \draw (0.25+8*\step,-\y*\step+0.05) node [anchor=east]{\tiny{$v_{\lab}$}};
\foreach[count=\x] \lab in {6,2,3, 5,4,1 , 0}
        \draw (\x*\step-0.05,0.2) node [anchor=north]{\tiny{$v_{\lab}$}};
	\draw [->,line width=2pt] (4,-0.6) -- +(0.45,0);
\end{tikzpicture} 
 & 
\begin{tikzpicture}
\def\step{0.43}
\draw (0,0) node [anchor = north west]  {\young(100,100,100,111,011,001,001)};
\foreach[count=\y] \lab in {0,1,4,5,3,2,6}
        \draw (0.25,-\y*\step+0.05) node [anchor=east]{\tiny{$v_{\lab}$}};
\foreach[count=\x] \lab in {8,7,9}
        \draw (\x*\step-0.05,0.2) node [anchor=north]{\tiny{$v_{\lab}$}};
\end{tikzpicture} 
\end{tabular}
    \end{center}
\end{example}

Let us now formally define parallelogram polyominoes as a particular type of 0/1 tableaux:

\begin{definition}
We call a 0/1 tableau $T$ a {\it{parallelogram polyomino}} 
of type $(m,n)$ if it satisfies the following:
\begin{enumerate}
\item[(i)] $T$ has of $m$ rows and $n$ columns.
\item[(ii)] The top left entry $T_{11}=1$ and the bottom right entry $T_{mn}=1$. 
\item[(iii)] There is at least one 1 in every row of $T$ and the 1s in a row are contiguous. 
\item[(iv)] The leftmost 1 in a row is weakly to the right of the leftmost 1 in the row above it.
\item[(v)] The rightmost 1 in a row is weakly to the right of the rightmost 1 in the row above it.
\end{enumerate}
Let $\Para_{m,n}$ be the set of all {\it{parallelogram polyomino}} of type $(m,n)$, and let $\Ribbon_{m,n} \subseteq \Para_{m,n}$ be the set of ribbon parallelogram  polyominoes.
\end{definition}

\begin{example}\label{pararib:example} \ \\
\begin{center}
\begin{tikzpicture}
\def\step{0.43}
\begin{scope}[xshift=0cm, yshift=0cm]
\draw (0,0) node [anchor = north west]  {\young(1100,1100,0110,0111,0011,0011)};
\draw (-0.5,-1.5) node {$T_1=$};
\draw (3,-1.5) node {$\in \Para_{6,4}$;};
\end{scope}
\begin{scope}[xshift=6cm, yshift=0cm]
\draw (0,0) node [anchor = north west]  {\young(1100,0100,0100,0110,0010,0011)};
\draw (-0.5,-1.5) node {$T_2=$};
\draw (3,-1.5) node {$\in \Ribbon_{6,4}$.};
\end{scope}
\end{tikzpicture} 
\end{center}
\end{example}

Let us now define the bounce path (or bounce polyomino even) of a parallelogram polyomino.

\begin{definition}
Let $P \in \Para_{m,n}$. Let $\bounce(P)$ be the ribbon parallelogram that is contained within $P$ as defined as follows.
Start at the cell in position (1,1), i.e. the top left cell that contains a 1. 
Move right until meeting the rightmost 1. Move down until meeting the lowest 1. Move right until meeting the rightmost 1, and so on until reaching the cell at position $(m,n)$. The result is $\bounce(P) \in \Rib_{m,n}$.
\end{definition}

\begin{example}
Consider the polyomino $T_1$ from Example~\ref{pararib:example}.
\begin{center}
\begin{tikzpicture}
\def\step{0.43}
\begin{scope}[xshift=0cm, yshift=0cm]
\draw (0,0) node [anchor = north west]  {\young(1100,0100,0100,0111,0001,0001)};
\draw[left] (-0.0,-1.5) node {$\bounce(T_1)=$};
\end{scope}
\end{tikzpicture}
\end{center}
\end{example}

We will now define labelled parallelogram polyominoes as row and column labelled versions of these tableaux.

\begin{definition}\label{def:dpp}
Let $T$ be a tableau consisting of $m$ rows and $n$ columns.
Let $\ell$ be a labelling of its rows and columns:
$$\ell = (\ell(\mathrm{row}_1),\ldots,\ell(\mathrm{row}_m),\ell(\mathrm{col}_1),\ldots,\ell(\mathrm{col}_n)).$$
We call a pair $D=(T,\ell)$ a {\it{labelled parallelogram polyomino}} 
of type $(m,n)$ if it satisfies the following:
\begin{enumerate}
\item[(i)] $T \in \Para_{m,n}$.
\item[(ii)] The labels of the $n$ columns are a permutation of the set $\{v_{m},\ldots,v_{m+n-1}\}$ and have the following property: 
	the labels of those columns whose topmost ones are at the same height are increasing from left to right.
\item[(iii)] The labels of the $m$ rows are a permutation of the set $\{v_0,v_1,\ldots,v_{m-1}\}$ and have the following property: 
	the label of the top row is $v_0$ and the labels of those rows whose leftmost ones are the same distance from the side are increasing from top to bottom.
\end{enumerate}
Let $\DecPara_{m,n}$ be the set of all {\it{labelled parallelogram polyomino}} of  type $(m,n)$, and let $\DecRib_{m,n} \subseteq \DecPara_{m,n}$ be the set of 
{\it{labelled ribbon parallelogram polyominoes}}.
\end{definition}

\begin{example} \label{bouncepara}
$D_1=(T_1,\ell_1)$ is a labelled parallelogram polyomino where $T_1$ is given in Example~\ref{pararib:example} 
and
\begin{align*}
\ell_1 =& (\ell_1(\mathrm{row}_1),\ldots,\ell_1(\mathrm{row}_m),\ell_1(\mathrm{col}_1),\ldots,\ell_1(\mathrm{col}_n))\\
=& (v_0,v_1,v_5,v_3,v_2,v_4,v_7,v_9,v_8,v_6).
\end{align*}
$D_2=(T_2,\ell_2)$ is a labelled ribbon parallelogram polyomino where $T_2$ is given in Example~\ref{pararib:example}
and $\ell_2=(v_0,v_2,v_4,v_5,v_3,v_1,v_6,v_8,v_7,v_9)$.
These labelled parallelogram polyominoes are illustrated below.
\begin{center}
\begin{tikzpicture}
\def\step{0.43}
\begin{scope}[xshift=0cm, yshift=0cm]
\draw (0,0) node [anchor = north west]  {\young(1100,1100,0110,0111,0011,0011)};
\foreach[count=\y] \lab in {0,1,5,3,2,4}
        \draw (0.25,-\y*\step+0.05) node [anchor=east]{\tiny{$v_{\lab}$}};
\foreach[count=\x] \lab in {7,9,8,6}
       \draw (\x*\step-0.05,0.2) node [anchor=north]{\tiny{$v_{\lab}$}};
\draw (3.2,-1.5) node {$\in \DecPara_{6,4}$;};
\draw [left] (-0.2,-1.5) node {$D_1=$};
\end{scope}
\begin{scope}[xshift=6cm, yshift=0cm]
\draw (0,0) node [anchor = north west]  {\young(1100,0100,0100,0110,0010,0011)};
\foreach[count=\y] \lab in {0,2,4,5,3,1}
        \draw (0.25,-\y*\step+0.05) node [anchor=east]{\tiny{$v_{\lab}$}};
\foreach[count=\x] \lab in {6,8,7,9}
        \draw (\x*\step-0.05,0.2) node [anchor=north]{\tiny{$v_{\lab}$}};
\draw (3.2,-1.5) node {$\in \DecRib_{6,4}$.};
\draw [left] (-0.2,-1.5) node {$D_2=$};
\end{scope}
\end{tikzpicture} 
\end{center}
\end{example}

\subsection{Rectangular decorated EW-tableaux}\label{sec:EWT}

EW-tableaux are 0/1 tableaux that were introduced in \cite{evw}.
Selig et al.~\cite{sss} used a slightly different definition for EW-tableaux wherein the singular row of `all 1s' must always be
the top row.
In this paper we will the latter definition along with two indexing conventions.

\begin{definition}[\cite{sss}]
\label{def:EWT}
A rectangular EW-tableau $T$ is a 0/1-filling of an $m$ row and $n$ column rectangle that satisfies the following properties:
\begin{enumerate}
\item[(i)] The top row of $T$ has a 1 in every cell.
\item[(ii)] Every other row has at least one cell containing a 0.
\item[(iii)] No four cells of $T$ that form the corners of a rectangle have 0s in two diagonally opposite corners and 1s in the other two.
\end{enumerate}
We denote by $EW_{m,n}$ the set of rectangular EW-tableaux having $m$ rows and $n$ columns.
\end{definition}

In Figure~\ref{fig:summary}, the set $\EW^{\textsf{Fer}}_{m,n}$ is the set of those members of $\EW_{m,n}$ whose 1s form a Ferrers shape in the top right corner.

\begin{example}\label{firsttabex}
Consider the following EW-tableau:
\begin{center}
\EWdiagramlab{11111,10011,00010,10011}{0,1,2,3}{4,5,6,7,8}{$T=$}
\end{center}
Examples of its entries are
$T_{4,5} 
=1$,
$T_{4,4} 
=1$, and
$T_{4,3} 
= 0$.
\end{example}

\begin{example}\label{exampleone}
For each of the following EW-tableaux, the row labels, from top to bottom, are $v_0,v_1,v_2$, and the column labels, read from left to right, are $v_3,v_4$. 
\begin{align*}
EW_{3,2} = \left\{ \scriptsize \young(11,01,01) , \young(11,10,10) , \young(11,01,00) , \young(11,00,01) , \young(11,10,00) , \young(11,00,10) , \young(11,00,00) \right\}.
\end{align*}
\end{example}

Before we introduce the notion of decorated EW-tableaux, 
we must introduce the notion of cornersupport 0s and 1s in EW-tableaux, 
as the decoration numbers depend on these quantities.

\begin{definition}\label{def:cornersupport}
Let $T\in \EW_{m,n}$.
Let us call a pair of cells $a$ and $b$ in a $T$ {\emph{non-attacking}} if they are in different rows and different columns.  
We also say that $b$ is non-attacking with respect to $a$ if $a$ and $b$ is a non-attacking pair, 
or simply that $b$ is non-attacking if it is clear from context what $a$ is.
We say that an entry $x \in \{0,1\}$ in $T_{jk}$ is a {\emph{cornersupport}} entry 
if and only if 
there exists a non-attacking $\mycomp{x}\not=x$ in $T_{j'k'}$ such that $T_{j'k}$ and $T_{jk'}$ both contain $x$.  
We say a cell is a {\it{cornersupport}} if it contains a cornersupport entry.
\end{definition}

\begin{example}\label{ex:rightangle}
Consider the following rectangular EW-tableau $T$.
The $1$ that appears at $T_{13}
$ is a \rightangle\ entry as is indicated by
the shaded entries. 
Similarly, the 0 that appears at position $T_{22}
$ is a cornersupport entry, 
this is validated by the non-attacking 1 at position $T_{34}
$ and the entries in the sub-square they induce.
\begin{center}
\begin{tikzpicture}
  \def\minix{0.21}
  \def\miniy{0.21}
\def\xx{0.435}
\def\bx{0}\def\by{0}
\def\initx{0.28}\def\inity{0.37}
\def\jnitx{0.31}\def\jnity{0.47}
\def\stepx{0.43}\def\stepy{0.43}
\fill[black!30!white] (-1.6525-\minix,-0.355-\miniy) rectangle (-1.6525+\minix,-0.355+\miniy);
\fill[black!30!white] (-1.6525-\minix,-1.65-\miniy) rectangle (-1.6525+\minix,-1.65+\miniy);
\fill[black!30!white] (-0.7875-\minix,-0.355-\miniy) rectangle (-0.7875+\minix,-0.355+\miniy);
\fill[black!30!white] (-0.7875-\minix,-1.65-\miniy) rectangle (-0.7875+\minix,-1.65+\miniy);
\draw (\bx,\by) node [anchor = north east]  {\young(1111,0010,0011,0010)};
\foreach \x/\y/\lab in {0/0/0,2/1/1,2/2/2,3/3/3}  
    \draw (\bx-\initx-4.75*\stepx,\by-\inity-\y*\stepy) node [anchor=west]{\tiny{$v_{\lab}$}};
\foreach \x/\y/\lab in {0/0/7,1/0/6,2/2/5,3/3/4}
    \draw (\bx-\jnitx-\x*\stepx,\by-\jnity+1.7*\stepy) node [anchor=north]{\tiny{$v_{\lab}$}}; 
\draw (-2.25,-1) node [anchor = east]{$T=$};
\end{tikzpicture}
\end{center}
\end{example}

Given $T\in \EW_{m,n}$ let
\begin{align*}
\eta_j(T) = \begin{cases}
\mbox{number of non-cornersupport 0s in row $v_j$,} & \mbox{ if $j<m$,}\\
\mbox{number of non-cornersupport 1s in column $v_j$,} & \mbox{ if $j\geq m$.}
\end{cases}
\end{align*}

\begin{definition}
A decorated rectangular EW-tableau of order $(m,n)$ is a pair $D=(T,a)$ where $T\in\EW_{m,n}$ 
and the sequence of integers $a=(a_1,\ldots,a_{m+n-1})$ satisfies $a_i \in \{0,1,\ldots,\eta_i(T)-1\}$ for all $1\leq i \leq m+n-1$.
Here the value $a_i$ is associated with the row/column having label $v_i$.
Let $\DEW_{m,n}$ be the set of all such decorated EW-tableaux of order $(m,n)$.
\end{definition}

\begin{example}\label{dewexample}
In the following tableau, we indicate in bold script those entries that {\underline{are}} cornersupport entries.
For the row with label $v_1$, there are three 0s, two of which are cornersupport 0s. The number of non-cornersupport 0s in this row is 1, so $\eta_1(T)=1$.
For row labelled $v_2$, there are two 0s that are both non-cornersupport 0s, so $\eta_2(T)=2$.
For row labelled $v_3$ we have $\eta_3(T) = 1$.
Turning now to the columns, we have $\eta_4(T)=1$ since there is a single 1 in that column and it is not a cornersupport 1. 
For the remainder we have $\eta_5(T)=1$, $\eta_6(T)=2$, and $\eta_7(T)=1$.
\begin{center}
\begin{tikzpicture}
  \def\minix{0.21}
  \def\miniy{0.21}
\def\xx{0.435}
\def\bx{0}\def\by{0}
\def\initx{0.25}\def\inity{0.37}
\def\jnitx{0.31}\def\jnity{0.47}
\def\stepx{0.43}\def\stepy{0.43}
\def\co{\bf{1}}
\def\cz{\bf{0}}
\draw (\bx,\by) node [anchor = north east]  {\young(11\co\co,\co\co10,00\co1,\cz\cz10)};
\foreach \x/\y/\lab in {0/0/0,2/1/1,2/2/2,3/3/3}  
    \draw (\bx-\initx-4.75*\stepx,\by-\inity-\y*\stepy) node [anchor=west]{\tiny{$v_\lab$}};
\foreach \x/\y/\lab in {0/0/7,1/0/6,2/2/5,3/3/4}
    \draw (\bx-\jnitx-\x*\stepx,\by-\jnity+1.7*\stepy) node [anchor=north]{\tiny{$v_\lab$}}; 
\draw (-2.25,-1) node [anchor = east]{$T=$};
\end{tikzpicture}
\end{center}
There are four decorated tableaux $D \in \DEW_{4,4}$ that have $T$ as the underlying EW-tableau, these are $(T,a)$ where 
\begin{align*}
a=(a_1,\ldots,a_7) \in & [0,0]\times[0,1] \times [0,0] \times [0,0] \times [0,0] \times [0,1] \times [0,0] \\
&= \{(0,0,0,0,0,0,0), (0,1,0,0,0,0,0), (0,0,0,0,0,1,0), (0,1,0,0,0,1,0)\} .
\end{align*}
These four decorated rectangular EW-tableaux are:
\begin{center}
\begin{tabular}{ll}
\begin{tikzpicture}
\def\minix{0.21}
\def\miniy{0.21}
\def\xx{0.435}
\def\bx{0}\def\by{0}
\def\initx{0.25}\def\inity{0.37}
\def\finitx{-0.20}\def\finity{-1.8}
\def\jnitx{0.31}\def\jnity{0.47}
\def\stepx{0.43}\def\stepy{0.43}
\draw (\bx,\by) node [anchor = north east]  {\young(1111,0010,0011,0010)};
\foreach \x/\y/\lab in {0/0/0,2/1/1,2/2/2,3/3/3}  
    \draw (\bx-\initx-4.75*\stepx,\by-\inity-\y*\stepy) node [anchor=west]{\tiny{$v_\lab$}};
\foreach \x/\y/\lab in {0/0/7,1/0/6,2/2/5,3/3/4}
    \draw (\bx-\jnitx-\x*\stepx,\by-\jnity+1.7*\stepy) node [anchor=north]{\tiny{$v_\lab$}}; 
\foreach \x/\y/\lab in {0/0/,2/1/0,2/2/0,3/3/0}  
    \draw (\finitx,\by-\inity-\y*\stepy) node [anchor=west]{\tiny{\bf\lab}};
\foreach \x/\y/\lab in {0/0/0,1/0/0,2/2/0,3/3/0}
    \draw (\bx-\jnitx-\x*\stepx,\finity) node [anchor=north]{\tiny{\bf\lab}}; 
\end{tikzpicture}
&
\begin{tikzpicture}
\def\minix{0.21}
\def\miniy{0.21}
\def\xx{0.435}
\def\bx{0}\def\by{0}
\def\initx{0.25}\def\inity{0.37}
\def\finitx{-0.20}\def\finity{-1.8}
\def\jnitx{0.31}\def\jnity{0.47}
\def\stepx{0.43}\def\stepy{0.43}
\draw (\bx,\by) node [anchor = north east]  {\young(1111,0010,0011,0010)};
\foreach \x/\y/\lab in {0/0/0,2/1/1,2/2/2,3/3/3}  
    \draw (\bx-\initx-4.75*\stepx,\by-\inity-\y*\stepy) node [anchor=west]{\tiny{$v_\lab$}};
\foreach \x/\y/\lab in {0/0/7,1/0/6,2/2/5,3/3/4}
    \draw (\bx-\jnitx-\x*\stepx,\by-\jnity+1.7*\stepy) node [anchor=north]{\tiny{$v_\lab$}}; 
\foreach \x/\y/\lab in {0/0/,2/1/0,2/2/1,3/3/0}  
    \draw (\finitx,\by-\inity-\y*\stepy) node [anchor=west]{\tiny{\bf\lab}};
\foreach \x/\y/\lab in {0/0/0,1/0/0,2/2/0,3/3/0}
    \draw (\bx-\jnitx-\x*\stepx,\finity) node [anchor=north]{\tiny{\bf\lab}}; 
\end{tikzpicture} \\
\begin{tikzpicture}
\def\minix{0.21}
\def\miniy{0.21}
\def\xx{0.435}
\def\bx{0}\def\by{0}
\def\initx{0.25}\def\inity{0.37}
\def\finitx{-0.20}\def\finity{-1.8}
\def\jnitx{0.31}\def\jnity{0.47}
\def\stepx{0.43}\def\stepy{0.43}
\draw (\bx,\by) node [anchor = north east]  {\young(1111,0010,0011,0010)};
\foreach \x/\y/\lab in {0/0/0,2/1/1,2/2/2,3/3/3}  
    \draw (\bx-\initx-4.75*\stepx,\by-\inity-\y*\stepy) node [anchor=west]{\tiny{$v_\lab$}};
\foreach \x/\y/\lab in {0/0/7,1/0/6,2/2/5,3/3/4}
    \draw (\bx-\jnitx-\x*\stepx,\by-\jnity+1.7*\stepy) node [anchor=north]{\tiny{$v_\lab$}}; 
\foreach \x/\y/\lab in {0/0/,2/1/0,2/2/0,3/3/0}  
    \draw (\finitx,\by-\inity-\y*\stepy) node [anchor=west]{\tiny{\bf\lab}};
\foreach \x/\y/\lab in {0/0/0,1/0/0,2/2/1,3/3/0}
    \draw (\bx-\jnitx-\x*\stepx,\finity) node [anchor=north]{\tiny{\bf\lab}}; 
\end{tikzpicture}
& 
\begin{tikzpicture}
\def\minix{0.21}
\def\miniy{0.21}
\def\xx{0.435}
\def\bx{0}\def\by{0}
\def\initx{0.25}\def\inity{0.37}
\def\finitx{-0.20}\def\finity{-1.8}
\def\jnitx{0.31}\def\jnity{0.47}
\def\stepx{0.43}\def\stepy{0.43}
\draw (\bx,\by) node [anchor = north east]  {\young(1111,0010,0011,0010)};
\foreach \x/\y/\lab in {0/0/0,2/1/1,2/2/2,3/3/3}  
    \draw (\bx-\initx-4.75*\stepx,\by-\inity-\y*\stepy) node [anchor=west]{\tiny{$v_\lab$}};
\foreach \x/\y/\lab in {0/0/7,1/0/6,2/2/5,3/3/4}
    \draw (\bx-\jnitx-\x*\stepx,\by-\jnity+1.7*\stepy) node [anchor=north]{\tiny{$v_\lab$}}; 
\foreach \x/\y/\lab in {0/0/,2/1/0,2/2/1,3/3/0}  
    \draw (\finitx,\by-\inity-\y*\stepy) node [anchor=west]{\tiny{\bf\lab}};
\foreach \x/\y/\lab in {0/0/0,1/0/0,2/2/1,3/3/0}
    \draw (\bx-\jnitx-\x*\stepx,\finity) node [anchor=north]{\tiny{\bf\lab}}; 
\end{tikzpicture}
\end{tabular}
\end{center}
\end{example}

In the paper ~\cite{dsss}, the definition of cornersupport entries was more involved, since it involved completing a general 
EW-tableau that is a Ferrers diagram in a particular manner to achieve a `supplementary tableau'.
The entries in the supplementary tableau were then used to determine whether a given tableau entry was cornersupport, or not.
Since we are dealing with rectangular EW-tableaux, the supplementary tableaux to which they correspond are the tableaux themselves.
This fact has allowed for shorter and more self-contained definitions in this subsection.
It allows us to offer the following interesting generalization of rectangular EW-tableaux.
Its generalisation to Ferrers shapes could be an interesting avenue of research to explore elsewhere.

\begin{definition}\label{mew:def}
We call $M=(T,a)$ a {\it{marked EW-tableau}} if $T \in EW_{m,n}$ and $a$ is a marking of non-cornersupport entries where 
\begin{enumerate}
\item[(i)] In each row of $M$ below the first row, the $(a_i+1)^{th}$ non-cornersupport 0 (from left) is marked with a $\star$. 
\item[(ii)] In each column of $M$ precisely, the $(a_i+1)^{th}$ non-cornersupport 1 (from the top) is marked with a $\star$. 
\end{enumerate}
Let $\MEW_{m,n}$ be the set of all such rectangular marked EW-tableaux.
\end{definition}

Not all entries of a row or column can be marked, only those that are non-cornersupport.
We will see in Subsection~\ref{see:ncs} that there is always at least one appropriate entry to mark so the above definition is well-defined.
Moreover, the definition above is equivalent to that of decorated rectangular EW-tableau, but has the advantage that 
the information regarding its structure is given by a marking of entries in the tableau rather than affixing a number the end 
of a row or bottom of a column.

\begin{example}
Consider the following EW-tableau $T \in \EW_{7,13}$.
\begin{center}
\begin{tikzpicture}
\def\step{0.43}
\draw (0,0) node [anchor = north west]  {\young(1111111111111,0011100000000,1011110111000,1011100000000,1011110111000,0011100000000,1011110111000)};
\foreach[count=\y] \lab in {0,1,2,3,4,5,6 }
        \draw (0.25,-\y*\step+0.05) node [anchor=east]{\tiny{$v_{\lab}$}};
\foreach[count=\x] \lab in {7,8,9,10,11,12,13,14,15,16,17,18,19}
        \draw (\x*\step-0.05,0.2) node [anchor=north]{\tiny{$v_{\lab}$}};
      \end{tikzpicture}
\end{center}
The following table indicates those entries that are non-cornersupport 1s and 0s.
\begin{center}
\begin{tikzpicture}
\def\step{0.43}
\def\zs{0^{\star}}
\draw (0,0) node [anchor = north west]  {\young(~1~~~~1~~~111,0~111~~~~~~~~,~0~~~10111000,1~~~~0~000~~~,~0~~~10111000,0~111~~~~~~~~,~0~~~10111000)};
\foreach[count=\y] \lab in {0,1,2,3,4,5,6 }
        \draw (0.25,-\y*\step+0.05) node [anchor=east]{\tiny{$v_{\lab}$}};
\foreach[count=\x] \lab in {7,8,9,10,11,12,13,14,15,16,17,18,19}
        \draw (\x*\step-0.05,0.2) node [anchor=north]{\tiny{$v_{\lab}$}};
      \end{tikzpicture}
\end{center}
We are now free to mark these non-cornersupport entries in the manner described in Definition~\ref{mew:def} and the following marked tableau is a member of $\MEW_{7,13}$.
\begin{center}
\begin{tikzpicture}
\def\step{0.43}
\def\zs{0^{\star}}
\def\os{1^{\star}}
\draw (0,0) node [anchor = north west]  {\young(1\os1111\os111\os\os\os,\zs01\os100000000,1011110\os11\zs00,\os011100\zs00000,101111011\os00\zs,\zs0\os1\os00000000,10111\os01\os1\zs00)};
\foreach[count=\y] \lab in {0,1,2,3,4,5,6 }
        \draw (0.25,-\y*\step+0.05) node [anchor=east]{\tiny{$v_{\lab}$}};
\foreach[count=\x] \lab in {7,8,9,10,11,12,13,14,15,16,17,18,19}
        \draw (\x*\step-0.05,0.2) node [anchor=north]{\tiny{$v_{\lab}$}};
      \end{tikzpicture}
\end{center}
The above marked EW-tableau corresponds to the decorated EW-tableaux $D=(T,a) \in \DEW_{7,13}$ where
$T$ is as given at the start of this example, and 
$$a=(0,2,1,4,0,2,0,0,1,0,1,2,0,0,2,1,0,0,0).$$
\end{example}


\section{A bijection from $\EW_{m,n}$ to $\DecRib_{m,n}$}

First we will define and prove the composition of the bijections $\bERm$ and $\bRmMr$ to yield $\bEMr$.

\begin{definition}
\label{myalg1}
Let $T \in \EW_{m,n}$.
Let the columns have labels (from left to right) $v_m,\ldots,v_{m+n-1}$. 
Label the rows of $T$ (from top to bottom) $v_0,\ldots,v_{m-1}$.
\begin{description}
\item[Step 1] Permute the columns of the EW-tableau $T$ such that there are only zeros to the left of every zero. For adjacent columns that are identical to one-another, ensure the labels are increasing from left to right.
\item[Step 2] Permute the rows of $T$ such that there are only ones above every one. For adjacent rows that are identical to one-another, sure the labels are increasing from top to bottom. 
Let $T'$ be the resulting tableau and let $\ell$ be the labelling of its rows and columns.
\item[Step 3] Let $R$ be an empty tableau having $m$ rows and $n$ columns. Let the rows and columns of $R$ have labelling $\ell$, i.e. the same as $T'$.
\item[Step 4] 
If $T'_{i,j} = 1$ and ($T'_{i+1,j} = 0$ or $i=m$) then set $R_{i,j}=1$.
If $T'_{i,j}=0$ and ($T'_{i,j+1} =1$ or $j=n$) then set $R_{i,j}=1$.
Fill all remaining unfilled  entries of $R$ with zeros.
\end{description}
Let the outcome of this procedure be $\mydecrib(T):=(R,\ell)$.
\end{definition}

\begin{proposition}\label{four:two}
$\mydecrib: \EW_{m,n} \mapsto \DecRib_{m,n}$.
\end{proposition}

\begin{proof}
Let $T \in \EW_{m,n}$. 
The initial labelling of the rows and columns of $T$ in Definition~\ref{myalg1} is consistent with the second point in Definition~\ref{def:dpp}.
Following the permuting of rows and columns of $T$ as outlined in steps 1 and 2 of Definition~\ref{myalg1} we achieve a tableau $T'$ 
such that the region of 1s is a Ferrers shape whose corner is in the top right of the tableau. Furthermore, the resulting labelling $\ell$
is consistent with the second point in Definition~\ref{def:dpp}. Note that since the top row of $T$ is the only row of all 1s, the same is true of $T'$ and its label remains unchanged as $v_0$.

In Step 3 of Definition~\ref{myalg1} we construct $R$ from $T'$ by looking at the 0/1 boundary in $T'$, and recording a shifted version of this as $R$. Since the boundary of the Ferrers shape of 1s is a path from the top left to the bottom right, $R$ will be a ribbon polyomino that begins at position $R_{11}$ and ends at $R_{mn}$.
\end{proof}

\begin{example}\label{four:three}
Consider the following EW-tableau $T \in \EW_{7,13}$.
\begin{center}
\begin{tikzpicture}
\def\step{0.43}
\draw (0,0) node [anchor = north west]  {\young(1111111111111,0011100000000,1011110111000,1011100000000,1011110111000,0011100000000,1011110111000)};
\foreach[count=\y] \lab in {0,1,2,3,4,5,6 }
        \draw (0.25,-\y*\step+0.05) node [anchor=east]{\tiny{$v_{\lab}$}};
\foreach[count=\x] \lab in {7,8,9,10,11,12,13,14,15,16,17,18,19}
        \draw (\x*\step-0.05,0.2) node [anchor=north]{\tiny{$v_{\lab}$}};
      \end{tikzpicture}
\end{center}
Apply Definition~\ref{myalg1}.
Step 1 tells us to permute the columns of the EW-tableau $T$ such that there are only zeros to the left of every zero, and to 
ensure the labels are increasing from left to right for adjacent columns that are identical to one-another.
Step 2 tells us to perform a similar permutation for rows of $T$ so that there are only ones above every one
and to ensure the labels are increasing from top to bottom for those adjacent rows that are identical to one-another.
The outcome of Steps 1 and 2 is the following tableau $T'$ and labelling $\ell$:
\begin{center}
\begin{tikzpicture}
\def\step{0.43}
\draw (0,0) node [anchor = north west]  {\young(1111111111111,0000011111111,0000011111111,0000011111111,0000000001111,0000000000111,0000000000111)};
\foreach[count=\y] \lab in {0,2,4,6,3,1,5 }
        \draw (0.25,-\y*\step+0.05) node [anchor=east]{\tiny{$v_{\lab}$}};
\foreach[count=\x] \lab in {8,13,17,18,19,12,14,15,16,7,9,10,11}
        \draw (\x*\step-0.05,0.2) node [anchor=north]{\tiny{$v_{\lab}$}};
      \end{tikzpicture}
\end{center}
Let $R$ be an empty tableau having $m$ rows and $n$ columns and whose row and column labels are the same as those for the tableau $T'$.
By applying step 4 to each of the entries of $T'$ and filling in the entries of $R$ accordingly, we have the following tableau and labelling $\mydecrib(T):=(R,\ell)$.

\begin{center}
\begin{tikzpicture}
\def\step{0.43}
\draw (0,0) node [anchor = north west]  {\young(1111100000000,0000100000000,0000100000000,0000111110000,0000000011000,0000000001000,0000000001111)};
\foreach[count=\y] \lab in {0,2,4,6,3,1,5  }
        \draw (0.25,-\y*\step+0.05) node [anchor=east]{\tiny{$v_{\lab}$}};
\foreach[count=\x] \lab in {8,13,17,18,19,12,14,15,16,7,9,10,11}
        \draw (\x*\step-0.05,0.2) node [anchor=north]{\tiny{$v_{\lab}$}};
      \end{tikzpicture}
\end{center}
\end{example}

\begin{example}
Three examples of members of $\EW_{3,2}$ and their images under $\mydecrib$:
\begin{center}
$T_1$=
\begin{tikzpicture}[baseline={([yshift=-.8ex]current bounding box.center)}]
\def\step{0.43}
\draw (0,0) node [anchor = north west]  {\young(11,00,00)};
\foreach[count=\y] \lab in {0,1,2}
 \draw (0.25,-\y*\step+0.05) node [anchor=east]{\tiny{$v_{\lab}$}};
\foreach[count=\x] \lab in {3,4}
        \draw (\x*\step-0.05,0.2) node [anchor=north]{\tiny{$v_{\lab}$}};
\end{tikzpicture} $\longrightarrow$
\begin{tikzpicture}[baseline={([yshift=-.8ex]current bounding box.center)}]
\def\step{0.43}
\draw (0,0) node [anchor = north west]  {\young(11,00,00)};
\foreach[count=\y] \lab in {0,1,2}
        \draw (0.25,-\y*\step+0.05) node [anchor=east]{\tiny{$v_{\lab}$}};
\foreach[count=\x] \lab in {3,4}
        \draw (\x*\step-0.05,0.2) node [anchor=north]{\tiny{$v_{\lab}$}};
      \end{tikzpicture} $\longrightarrow$
  \begin{tikzpicture}[baseline={([yshift=-.8ex]current bounding box.center)}]
\def\step{0.43}
\draw (0,0) node [anchor = north west]  {\young(11,01,01)};
\foreach[count=\y] \lab in {0,1,2}
        \draw (0.25,-\y*\step+0.05) node [anchor=east]{\tiny{$v_{\lab}$}};
\foreach[count=\x] \lab in {3,4}
        \draw (\x*\step-0.05,0.2) node [anchor=north]{\tiny{$v_{\lab}$}};
      \end{tikzpicture}$=\mydecrib(T_1)$\\
$T_2$=
\begin{tikzpicture}[baseline={([yshift=-.8ex]current bounding box.center)}]
\def\step{0.43}
\draw (0,0) node [anchor = north west]  {\young(11,00,01)};
\foreach[count=\y] \lab in {0,1,2}
\draw (0.25,-\y*\step+0.05) node [anchor=east]{\tiny{$v_{\lab}$}};
\foreach[count=\x] \lab in {3,4}
\draw (\x*\step-0.05,0.2) node [anchor=north]{\tiny{$v_{\lab}$}};       
\end{tikzpicture} $\longrightarrow$
\begin{tikzpicture}[baseline={([yshift=-.8ex]current bounding box.center)}]
\def\step{0.43}
\draw (0,0) node [anchor = north west]  {\young(11,01,00)};
\foreach[count=\y] \lab in {0,2,1}
        \draw (0.25,-\y*\step+0.05) node [anchor=east]{\tiny{$v_{\lab}$}};
\foreach[count=\x] \lab in {3,4}
        \draw (\x*\step-0.05,0.2) node [anchor=north]{\tiny{$v_{\lab}$}};
      \end{tikzpicture} $\longrightarrow$
  \begin{tikzpicture}[baseline={([yshift=-.8ex]current bounding box.center)}]
\def\step{0.43}
\draw (0,0) node [anchor = north west]  {\young(10,11,01)};
\foreach[count=\y] \lab in {0,2,1}
        \draw (0.25,-\y*\step+0.05) node [anchor=east]{\tiny{$v_{\lab}$}};
\foreach[count=\x] \lab in {3,4}
        \draw (\x*\step-0.05,0.2) node [anchor=north]{\tiny{$v_{\lab}$}};
      \end{tikzpicture}$=\mydecrib(T_2)$\\
$T_3$=
\begin{tikzpicture}[baseline={([yshift=-.8ex]current bounding box.center)}]
\def\step{0.43}
\draw (0,0) node [anchor = north west]  {\young(11,10,10)};
\foreach[count=\y] \lab in {0,1,2}
\draw (0.25,-\y*\step+0.05) node [anchor=east]{\tiny{$v_{\lab}$}};
\foreach[count=\x] \lab in {3,4}
\draw (\x*\step-0.05,0.2) node [anchor=north]{\tiny{$v_{\lab}$}};
\end{tikzpicture} $\longrightarrow$
\begin{tikzpicture}[baseline={([yshift=-.8ex]current bounding box.center)}]
\def\step{0.43}
\draw (0,0) node [anchor = north west]  {\young(11,01,01)};
\foreach[count=\y] \lab in {0,1,2}
        \draw (0.25,-\y*\step+0.05) node [anchor=east]{\tiny{$v_{\lab}$}};
\foreach[count=\x] \lab in {4,3}
        \draw (\x*\step-0.05,0.2) node [anchor=north]{\tiny{$v_{\lab}$}};
      \end{tikzpicture} $\longrightarrow$
 \begin{tikzpicture}[baseline={([yshift=-.8ex]current bounding box.center)}]
\def\step{0.43}
\draw (0,0) node [anchor = north west]  {\young(10,10,11)};
\foreach[count=\y] \lab in {0,1,2}
        \draw (0.25,-\y*\step+0.05) node [anchor=east]{\tiny{$v_{\lab}$}};
\foreach[count=\x] \lab in {4,3}
        \draw (\x*\step-0.05,0.2) node [anchor=north]{\tiny{$v_{\lab}$}};
      \end{tikzpicture}$=\mydecrib(T_3)$
\end{center}
\end{example}

\begin{theorem}\label{first:injective}
$\mydecrib : \EW_{m,n} \mapsto \DecRib_{m,n}$ is injective.
\end{theorem}

\begin{proof}
Let $A,B\in \EW_{m,n}$ and set $\mydecrib(A)=(T^{(A)},\ell^{(A)})$ and $\mydecrib (B)=(T^{(B)},\ell^{(B)})$.
Suppose $\mydecrib(A)=\mydecrib(B)$.
Then $(T^{(A)},\ell^{(A)}) = (T^{(B)},\ell^{(B)})$ which implies $T^{(A)}=T^{(B)}$ and $\ell^{(A)}=\ell^{(B)}$.

Suppose $A^{(2)}$ and $B^{(2)}$ are the result of applying Steps 1 and 2 of Definition~\ref{myalg1} to $A$ and $B$, respectively. 
Then $A^{(2)}$ and $B^{(2)}$ are tableaux both having $m$ rows and $n$ columns such that the ones form a Ferrers diagram in the top right of each tableaux, 
and the zeros form a Ferrers diagram in the bottom left of the tableaux.

Let  $R^{(A)}$ and $R^{(B)}$ be the empty tableaux mentioned in Step 3 of Definition~\ref{myalg1}.
Note that they both have $m$ rows and $n$ columns.

The tableaux $R^{(A)}$ and $R^{(B)}$ are filled according to Step 4 of Definition~\ref{myalg1} as it is applied to $A^{(2)}$ and $B^{(2)}$, respectively. 
Completing Step 4 gives us $\mydecrib(A) = R^{(A)} = (T^{(A)},\ell^{(A)})$ and $\mydecrib(B) = R^{(B)} = (T^{(B)},\ell^{(B)})$.

Let us now note that the operation in Step 4 of Definition~\ref{myalg1} is invertible.
This is seen through the following procedure:
in order to recover $A^{(2)}$ from $T^{(A)}$ we do as follows:
\begin{itemize}
\item In $T^{(A)}$, change every 0 in the first row to a 1.
\item For each of the other rows of $T^{(A)}$: replace the leftmost 1 with a 0 and change all entries 
to its right 
(which, by definition, are all 0s) 
to 1.
\item The outcome of doing this produces $A^{(2)}$.
\end{itemize}
Since $T^{(B)}=T^{(A)}$, by assumption, applying these same rules to $T^{(B)}$ will result in the same tableau $B^{(2)} = A^{(2)}$.

So in assuming that $\mydecrib(A) = \mydecrib(B)$, we must have $A^{(2)}=B^{(2)}$. 
In the original construction, Definition~\ref{myalg1}, in order to get from $A$ to $A^{(2)}$ we permuted the rows and columns of $A$ with respect to $\ell^{(A)}$.
In order to recover $A$ from $A^{(2)}$ we reorder the rows and columns of $A^{(2)}$ with respect to the inverse of the permutation of labels given by $\ell^{(A)}$.
Since, by assumption, $\ell^{(A)}=\ell^{(B)}$, and combining this with the (derived) fact $A^{(2)}=B^{(2)}$, doing the same to $B^{(2)}$ will result in the same tableau $A$.

This implies that $A=B$ and so the mapping $\mydecrib$ is injective.
\end{proof}

\begin{theorem} \label{first:surjective}
$\mydecrib : \EW_{m,n} \mapsto \DecRib_{m,n}$  is surjective.
\end{theorem}

\begin{proof}
Given $(T,\ell)\in \DecRib_{m,n}$, define the $m\times n$ tableau 
\begin{align*}
B_{ij} := &
	\begin{cases}
	1 & \mbox{ if } (i,j)=(1,1), \\
	1 & \mbox{ if there exists } k<j \mbox{ s.t. } T_{ik}=1, \\
	0 & \mbox{ otherwise.}
	\end{cases}
\end{align*}
Let $A$ be the tableau with $m$ rows and $n$ columns that is achieved through the following simple operation on $(B,\ell)$:
\begin{itemize}
\item permute the columns of $(B,\ell)$ so that the column labels are increasing from left to right
\item permute the rows of $(B,\ell)$ so that the row labels are increasing from top to bottom.
\end{itemize}
Now let us consider $\mydecrib(A)$. 
By Definition~\ref{myalg1}, to construct $\mydecrib(A)$ we first apply Steps 1 and 2 to $A$. 
We achieve a tableau $A^{(2)}$ such that there are only zeros to the left of every zero, and ones above every one. 
In order to recover $A$ from $A^{(2)}$ we reorder the rows and columns of $A^{(2)}$ with respect to the inverse of the permutation of labels given by $\ell$.
Since, by assumption, we get $A$ from $B$. Therefore $A^{(2)}=B$.

According to Step 4 of Definition~\ref{myalg1}, when it is is applied to $B$ it will result in $(T,\ell)$. 
The reason for this is because Step 4 of Definition~\ref{myalg1} is invertible. 
The following observation shows this: in order to recover $A^{(2)}$ from $T$ we do as follows:
\begin{itemize}
\item In $T$, change every 0 in the first row to a 1.
\item For every other row of $T$: change the leftmost 1 to 0 and simultaneously change all 0s to the right of that leftmost 1 to 1.
\item The outcome of doing this gives $A^{(2)}$.
\end{itemize}
This implies that there exists $A\in \EW_{m,n}$ such that $\mydecrib(A)=(T,\ell)$. 
\end{proof}

\begin{example}
Let $(T,\ell)$ be the labelled ribbon polyomino on the left in Figure~\ref{amal:fig:one}.
The tableau in the centre is constructed from $(T,\ell)$ by changing every 0 in the first row to a 1, 
then in all other rows change every leftmost 1 to 0 and simultaneously change all 0s to the right of that 1 to 1. 
The $EW$ tableau on the right is constructed by permuting the rows and columns of the centre tableau to be in increasing order from top to bottom and right to left, respectively.
\begin{figure}
\label{amal:fig:one}
\begin{center}
 \begin{tikzpicture}[baseline={([yshift=-.8ex]current bounding box.center)}]
\def\step{0.43}
\draw (0,0) node [anchor = north west]  {\young(11100,00100,00111,00001)};
\foreach[count=\y] \lab in {0,2,3,1}
        \draw (0.25,-\y*\step+0.05) node [anchor=east]{\tiny{$v_{\lab}$}};
\foreach[count=\x] \lab in {4,6,8,5,7}
        \draw (\x*\step-0.05,0.2) node [anchor=north]{\tiny{$v_{\lab}$}};
      \end{tikzpicture}$\longrightarrow$ \begin{tikzpicture}[baseline={([yshift=-.8ex]current bounding box.center)}]
\def\step{0.43}
\draw (0,0) node [anchor = north west]  {\young(11111,00011,00011,00000)};
\foreach[count=\y] \lab in {0,2,3,1}
        \draw (0.25,-\y*\step+0.05) node [anchor=east]{\tiny{$v_{\lab}$}};
\foreach[count=\x] \lab in {4,6,8,5,7}
        \draw (\x*\step-0.05,0.2) node [anchor=north]{\tiny{$v_{\lab}$}};
      \end{tikzpicture}$\longrightarrow$\begin{tikzpicture}[baseline={([yshift=-.8ex]current bounding box.center)}]
\def\step{0.43}
\draw (0,0) node [anchor = north west]  {\young(11111,00000,01010,01010)};
\foreach[count=\y] \lab in {0,1,2,3}
        \draw (0.25,-\y*\step+0.05) node [anchor=east]{\tiny{$v_{\lab}$}};
\foreach[count=\x] \lab in {4,5,6,7,8}
        \draw (\x*\step-0.05,0.2) node [anchor=north]{\tiny{$v_{\lab}$}};
      \end{tikzpicture}
\end{center}
\caption{An example of an element of $\DecRib_{4,5}$ (left) and its corresponding $\EW_{4,5}$-tableau (right).
}
\end{figure}
\end{example}
Theorems~\ref{first:injective} and \ref{first:surjective} can now be combined to yield:
\begin{theorem}\label{three:eight}
$\mydecrib : \EW_{m,n} \mapsto \DecRib_{m,n}$ is a bijection.
\end{theorem}
\newcommand{\mydecribinv}{\psi}
The inverse of the bijection $\mydecrib$ is the following:
\begin{definition}\label{myalg2}\ \\
Let $(R,\ell) \in \DR_{m,n}$. Let $T'$ be a tableau having the same dimensions as $R$ with
$$T'_{ij} = 
	\begin{cases}
	1 & \mbox{ if } (i,j)=(1,1), \\
	1 & \mbox{ if there exists } k<j \mbox{ s.t. } R_{ik}=1, \\
	0 & \mbox{ otherwise.}
	\end{cases}
		$$
Let $T$ be the tableau one gets from $T'$ by performing the following simple operations:
permute the columns of $T'$ so that the labels $(v_{m},\ldots,v_{m+n-1})$ are increasing from left to right 
and permute the rows of $T'$ so that labels $(v_{0},\ldots,v_{m-1})$ are increasing from top to bottom.
Let the outcome of this procedure be $\mydecribinv (R,\ell) := T$.
\end{definition}

\begin{example}
Let $(R,\ell)$ be the following labelled ribbon parallelogram polyomino:
\\
\begin{center}
\EWdiagramlab{1100,0111,0001}{0,2,1}{4,6,3,5}{$(R,\ell)=$}
\end{center}
Apply the rule to construct $T'$:\\
\centerline{
\EWdiagramlab{1111,0011,0000}{0,2,1}{4,6,3,5}{$T'=$}
}
Permute the columns of $T'$ so that the column labels are increasing from left to right $(v_{3},v_{4},v_{5},v_{6})$, 
and permute the rows of $T'$ so that the row labels are increasing from top to bottom $(v_{0},v_{1},v_{2})$ to get the following member of $\EW_{3,4}$:\\
\centerline{
\EWdiagramlab{1111,0000,1010}{0,1,2}{3,4,5,6}{$T=$}
}
\end{example}

\mynewpage
\section{A bijection from $\MEW_{m,n}$ to $\DecPara_{m,n}$}

\newcommand{\surplus}{\mathsf{surp}}
\newcommand{\surp}{\surplus}
\newcommand{\lrib}[1]{\langle #1 \rangle}

Before we present the main bijection in this section, we need to do two things.
The first is to introduce an alternative notation for specifying labelled parallelogram polyominoes.
The second is to show the reader how to `see' non-cornersupport entries in an EW-tableaux.

\subsection{Expanding labelled ribbon polyominoes}
When introduced in Section 2, labelled parallelogram polyominoes were specified as a pair $(T,\ell)$ where 
$T$ is an unlabelled parallelogram polyomino, and $\ell$ is a labelling of its rows and columns that respects a simple convention given 
in Definition~\ref{def:dpp}.

The alternative notation we will now use is as follows.
We will specify a labelled ribbon parallelogram polyomino $D$ as a pair $\lrib{D',a}$ 
where $D'$ is the labelled ribbon parallelogram polyomino 
that is the bounce path of $D$, and $a$ is a sequence that details how many cells to expand the bounce path by in all rows and columns in order to expand the ribbon into a parallelogram polyomino.
There are two important points to be made in relation to this notation:

\begin{itemize}
\item
The number of cells one may add to individual rows and columns of a labelled ribbon parallelogram polyomino is, of course, limited. 
We cannot add so many cells to a column or row that the underlying bounce path is different. 
(This would destroy the uniqueness property of the notation.)
Thus certain integer sequences `support' a labelled ribbon parallelogram polyomino in that they preserve the bounce path.
\item Once cells are added to a row or column, it may not immediately resemble a parallelogram polyomino.
There is no harm in this since the addition of such cells will partition those rows/columns are detailed in the second and third points of Definition~\ref{def:dpp}.
To overcome this we simply shuffle these contiguous rows and columns so that `increasing' property is satisfied.
\end{itemize}

\begin{proposition}\label{dp:to:dr}
Let $D=(T,\ell) \in \DecPara_{m,n}$ with $\ell = (\ell_1,\ldots,\ell_{m+n-1})$.
Then $D$ can be uniquely written as a pair $\lrib{D',\surplus}$ where 
$D'=(\bounce(T),\ell ') \in \DecRib_{m,n}$ is a labelled ribbon parallelogram polyomino and 
$\surplus = (\surplus_1,\ldots,\surplus_{m+n-1})$ is a sequence detailing the number 
of cells to be appended to a row/column in $D'$ in order to achieve to achieve $D$.
There are two cases depending on whether the surplus value refers to a row or column label:
\begin{description}
\item[Case $i<m$] Let $x$ be the row of $T$ such that $\ell_x = v_i$.  
Let $y$ be the index such that $T_{xy}$ is the leftmost 1 in that row in $T$.
Let $x'$ be such that $\ell'_{x'}=v_i$ and let $y'$ be the index such 
that $\bounce(T)_{x'y'}$ is the index of the leftmost 1 in that row in $\bounce(T)$.
Define $\surplus_i := x'-x$. This number represents the number of cells to be added to the left of $\bounce(T)_{x'y'}$.
\item[Case $i\geq m$] Let $x$ be the column of $T$ such that $\ell_x=v_i$.
Let $y$ be the index such that $T_{yx}$ is the topmost 1 in that column in $T$.
Let $x'$ be such that $\ell'_{x'} = v_i$ and let $y'$ be the index such 
that $\bounce(T)_{y'x'}$ is the index of the topmost 1 in that column in $\bounce(T)$.
Define $\surplus_i := y'-y$.
This number represents the number of cells to be added above $\bounce(T)_{y'x'}$.
\end{description}
\end{proposition}

\begin{proof}
The construction in this proposition uses the observation that, 
while a labelled parallelogram polyomino is a parallelogram polyomino complete with a labelling of its rows and columns that obeys the labelling convention of Def.~\ref{def:dpp},
one may first specify the labelled bounce path (itself a labelled ribbon parallelogram polyomino) of the labelled parallelogram polyomino. 
The only extra information is the number of cells by which to extend every row and column 
while ensuring the bounce path of the parallelogram polyomino remains unchanged.
\end{proof}

\begin{example}\label{decpara:to:decrib}
Let us consider the labelled parallelogram polyomino where the underlying parallelogram polyomino is given in Example~\ref{pararib:example}.
We have:
\begin{center}
\begin{tikzpicture}
\def\step{0.43}
\begin{scope}[xshift=0cm, yshift=0cm]
\draw (0,0) node [anchor = north west]  {\young(1100,1100,0110,0111,0011,0011)};
\foreach[count=\y] \lab in {0,4,1,3,2,5}
        \draw (0.25,-\y*\step+0.05) node [anchor=east]{\tiny{$v_{\lab}$}};
\foreach[count=\x] \lab in {6,8,9,7}
        \draw (\x*\step-0.05,0.2) node [anchor=north]{\tiny{$v_{\lab}$}};
\draw[left] (-0.22,-1.5) node {$D=(T,\ell)=$};
\draw (3,-1.5) node {$\in \DecPara_{6,4}$;};
\end{scope}
\begin{scope}[xshift=9cm, yshift=0cm]
\draw (0,0) node [anchor = north west]  {\young(1100,0100,0100,0111,0001,0001)};
\foreach[count=\y] \lab in {0,1,3,4,2,5}
        \draw (0.25,-\y*\step+0.05) node [anchor=east]{\tiny{$v_{\lab}$}};
\foreach[count=\x] \lab in {6,8,7,9}
        \draw (\x*\step-0.05,0.2) node [anchor=north]{\tiny{$v_{\lab}$}};
\draw[left] (-0.22,-1.5) node {$D'=(\bounce(T),\ell')=$};
\draw (3,-1.5) node {$\in \DecRib_{6,4}$};
\end{scope}
\end{tikzpicture}
\end{center}
Here $D=(T,\ell)$ where $\ell=(v_4,v_1,v_3,v_2,v_5,v_6,v_8,v_9,v_7)$ and the corresponding labelled ribbon parallelogram polyomino is 
$D'=(\bounce(T),\ell')$ where $\ell'=(v_1,v_3,v_4,v_2,v_5,v_6,v_8,v_7,v_9)$.
We may now write $D$ as the pair $D=\lrib{D',\surplus}$ where
\begin{align*}
\surplus=& (2-2,4-3,2-2,2-1,4-3,1-1,4-4,1-1,4-3)\\
=& (0,1,0,1,1,0,0,0,1).
\end{align*}
\end{example}

\subsection{How to discern non-cornersupport entries in an EW-tableau}\label{see:ncs}

Given a large rectangular EW-tableaux, it may seem quite a daunting task to decide which entries are non-cornersupport using only Definition~\ref{def:cornersupport}. 
To overcome this, we will illustrate a simple transformation of a rectangular tableaux so that all non-cornersupport entries can be immediately discerned.
In order to do this, we will make use of the observation that an entry in a rectangular EW-tableau is a non-cornersupport entry iff it is a non-cornersupport entry of a tableau that results from permuting the rows and columns in any order.

Given an EW-tableau $T \in \EW_{m,n}$, let $T'$ be the result of applying Steps 1 and 2 of Definition~\ref{myalg1} to $T$.
A more descriptive way to say this is that $T'$ is the reordering of rows and columns so that the 1s of the tableau form a Ferrers shape whose corner is at the top right of the tableau.

\begin{example}
Consider the tableau $T$ from Example~\ref{four:three}.
\begin{center}
\begin{tikzpicture}
\def\step{0.43}
\draw (0,0) node [anchor = north west]  {\young(1111111111111,0011100000000,1011110111000,1011100000000,1011110111000,0011100000000,1011110111000)};
\foreach[count=\y] \lab in {0,1,2,3,4,5,6 }
        \draw (0.25,-\y*\step+0.05) node [anchor=east]{\tiny{$v_{\lab}$}};
\foreach[count=\x] \lab in {7,8,9,10,11,12,13,14,15,16,17,18,19}
        \draw (\x*\step-0.05,0.2) node [anchor=north]{\tiny{$v_{\lab}$}};
      \end{tikzpicture}
\end{center}
Permute the rows and columns as outlined in Steps 1 and 2 of Definition~\ref{myalg1} to get:

\centerline{\EWdiagramlab{1111111111111,0000011111111,0000011111111,0000011111111,0000000001111,0000000000111,0000000000111}{0,2,4,6,3,1,5}{8,13,17,18,19,12,14,15,16,7,9,10,11}{$T'=$}}

A 1 is a cornersupport one if there exists a sub-square such that there is a non-attacking 0 and the two other entries are 1s.
With this reordered tableau, such a 0 would necessarily be south west of the 1.
Consider the 1 at position $(v_0,v_7)$. 
We can see that there is a 1 to its left and position $(v_0,v_{19})$ and a 1 beneath it at position $(v_3,v_7)$.
The remaining entry in the square induced by these three 1s is $T'(v_3,v_{19})=0$. Thus the entry $T'(v_0,v_7)=1$ is a cornersupport 1. 
Any 1s to the left of that one will be cornersupport ones for that same reason.

However, if we examine the entry $T'(v_2,v_{14})=1$, it is contained in a rectangle of 1s that borders the region of 0s.
In other words, if we select any 1 to its left and any 1 below it, then the missing entry in the square induced by these three entries is also a 1. So $T'(v_2,v_{14})=1$ is a non-cornersupport 1.

Thus 1s in these special rectangles that border the 0 region are non-cornersupport 1s. Precisely the same is true of 0s that are close to the diagonal, and the non-cornersupport 1s and 0s are highlighted in the shaded region in the following diagram:

\begin{center}
\begin{tikzpicture}
  \def\minix{0.21}
  \def\miniy{0.21}
\def\xx{0.435}
\def\bx{0}\def\by{0}
\def\initx{0.28}\def\inity{0.37}
\def\jnitx{0.31}\def\jnity{0.47}
\def\stepx{0.43}\def\stepy{0.43}
\fill[black!30!white] (0.14,-0.14) rectangle (0.14+5*0.431,-0.14-4*0.43);
\fill[black!30!white] (0.14+5*0.431,-0.14-0.43) rectangle (0.14+9*0.431,-0.14-5*0.43);
\fill[black!30!white] (0.14+9*0.431,-0.14-4*0.431) rectangle (0.14+10*0.431,-0.14-7*0.43);
\fill[black!30!white] (0.14+10*0.431,-0.14-5*0.431) rectangle (0.14+13*0.431,-0.14-7*0.43);
\EWdiagramlabwithin{1111111111111,0000011111111,0000011111111,0000011111111,0000000001111,0000000000111,0000000000111}{0,2,4,6,3,1,5}{8,13,17,18,19,12,14,15,16,7,9,10,11}{$T'=$}
\end{tikzpicture}
\end{center}
We can now use this labelled diagram to shade in the corresponding 1s and 0s in $T$.
One other important fact regarding this transformation is that the order in which pairs of non-cornersupport 1s (resp. 0s) appear in relation to one another from left to right (resp. top to bottom) is preserved.
\end{example}

\subsection{A bijection from $\MEW$ to $\DecPara$}
\newcommand{\mydecpara}{\Phi}
With this new notation for labelled parallelogram polyominoes, we are now in a position to define the main bijection of this section.
Given a marked rectangular EW-tableau $M=(T,a) \in \MEW_{m,n}$, 
we will now show how to construct a unique labelled parallelogram polyomino $D=\mydecpara(M)$.

\begin{definition}\label{map:one}
Let $M=(T,a) \in \MEW_{m,n}$ and let $\mydecrib(T)=(R,\ell)$, where $\mydecrib$ is given in Definition~\ref{myalg1}.
Define $D=\mydecpara(M) := \lrib{\mydecrib(T),\zeta}$ where $\zeta=(\zeta_1,\ldots,\zeta_{m+n-1})$ and $\zeta_i=\eta_i(T)-a_i-1$.
\end{definition}

\begin{example}\label{second:last}
Consider the following marked EW-tableaux
\begin{center}
\EWdiagramlab{1\os1111\os111\os\os\os,\zs01\os100000000,10111101\os100\zs,\os011100\zs00000,101111\zs\os1\os000,\zs0\os1100000000,10111\os01110\zs0}{0,1,2,3,4,5,6}{7,8,9,10,11,12,13,14,15,16,17,18,19}{$M=(T,a)=$}
\end{center}
Here $a=(0,4,1,1,0,3,0,0,1,0,1,2,0,1,0,1,0,0,0)$.
The corresponding labelled ribbon parallelogram polyomino is
\begin{center}
\EWdiagramlab{1111100000000,0000100000000,0000100000000,0000111110000,0000000011000,0000000001000,0000000001111}{0,2,4,6,3,1,5}{8,13,17,18,19,12,14,15,16,7,9,10,11}{$D'=(R,\ell)=\mydecrib(T)=$}
\end{center}
The vector $\eta(T) = (1,5,4,5,1,5,1,1,2,2,2,3,1,3,3,3,1,1,1)$.
Thus the vector $\zeta = \eta(T) - a - 1$ is 
$$\zeta = (0,0,2,3,0,1,0,0,0,1,0,0,0,1,2,1,0,0,0) .$$
We now use the construction for the pair $\lrib{R,\zeta}$ in Proposition~\ref{dp:to:dr} to append extra cells to the bounce path polyomino, and reorder to get a labelled parallelogram polyomino.
Let us illustrate this one step as three different steps so that it is clear. 
First we append extra cells to the left of the leftmost cells in a bounce path. These are illustrated using X's in the following. 
Since $\zeta_4=3$, we add three X cells to the left of the one in row labelled $v_4$.
Next we append extra cells above the topmost cells in a bounce paths. These are illustrated using Y's in the following diagram.
Since $\zeta_{14}=1$ we add one Y above the topmost one in the column having label $v_{14}$. 
\begin{center}
\EWdiagramlab{11111~~~~~~~~,~~~~1~~Y~~~~~,~XXX1~YYY~~~~,~~~X11111~~~~,~~~~~~XX11~~~,~~~~~~~~~1~Y~,~~~~~~~~~1111}{0,2,4,6,3,1,5}{8,13,17,18,19,12,14,15,16,7,9,10,11}{}
\end{center}
Within those contiguous 1s in the bounce path that have the same abscissa, we reorder these so that the number of X's in a row is weakly decreasing from top to bottom. 
We do the same for those contiguous 1s that have the same ordinate in relation to the Y's. 
This achieves the following reordering.
\begin{center}
\EWdiagramlab{11111~~~~~~~~,~XXX1Y~~~~~~~,~~~X1YYY~~~~~,~~~~11111~~~~,~~~~~~XX11~~~,~~~~~~~~~1Y~~,~~~~~~~~~1111}{0,4,6,2,3,1,5}{8,13,17,18,19,15,14,16,12,7,10,9,11}{}
\begin{tikzpicture}
	\draw [white] (-0.30,0) rectangle (0.45,3);
	\draw [->,line width=2pt] (0,1.5) -- +(0.45,0);
\end{tikzpicture} 
\EWdiagramlab{11111~~~~~~~~,~11111~~~~~~~,~~~11111~~~~~,~~~~11111~~~~,~~~~~~1111~~~,~~~~~~~~~11~~,~~~~~~~~~1111}{0,4,6,2,3,1,5}{8,13,17,18,19,15,14,16,12,7,10,9,11}{}
\end{center}
The final step is replacing X's and Y's with 1s. 
This results in the following member of $\DecPara_{7,13}$:
\begin{center}
\EWdiagramlab{1111100000000,0111110000000,0001111100000,0000111110000,0000001111000,0000000001100,0000000001111}{0,4,6,2,3,1,5}{8,13,17,18,19,15,14,16,12,7,10,9,11}{$\mydecpara(M) = $}
\end{center}
\end{example}

\begin{theorem}
$\mydecpara: \MEW_{m,n} \to \DecPara_{m,n}$ is a bijection.
\end{theorem}

\begin{proof} 
The mapping $\mydecpara$ is well-defined. This is seen through the following consideration.
Let $M=(T,a) \in \MEW_{m,n}$ and $D=\mydecpara(T,a)$. 
The first step, that of constructing $\mydecrib(T)$, in making $D$ was shown to be a well defined in Proposition~\ref{four:two}, so $\mydecrib(T) \in \DecRib_{m,n}$.
The second step in constructing $D$ is enlarging $\mydecrib(T)$ with respect to the values given in $\eta$. 
The construction in subsection~\ref{see:ncs} shows that the maximal additions to each row and column that may occur preserve the bounce path.
Thus $\mydecpara(T,a) \in \DecPara_{m,n}$.\ \\[1em]
Let $M_1=(T_1,a)$ and $M_2=(T_2,b)$ be in $\MEW_{m,n}$. 
Let $D_1 = \mydecpara(T_1,a)$ and $D_2 = \mydecpara(T_2,b)$. 
Suppose that $D_1 = D_2$.
From Definition~\ref{map:one} we see that 
$D_1=\lrib{\phi(T_1),\zeta^{(a)}}$ where $\zeta^{(a)}_i = \eta_i(T_1)-a_i-1$.
Also, 
$D_2=\lrib{\phi(T_2),\zeta^{(b)}}$ where $\zeta^{(b)}_i = \eta_i(T_2)-b_i-1$.
As $D_1=D_2$ we must have that $\lrib{\phi(T_1),\zeta^{(a)}} = \lrib{\phi(T_2),\zeta^{(b)}}$. 
The construction $\lrib{\cdot,\cdot}$ has as first argument a unique labelled ribbon parallelogram polyomino, 
and as second argument the information detailing the number of cells to be added to each row/column.
Because of this, we must have $\phi(T_1) = \phi(T_2)$, i.e. the labelled parallelogram polyominoes have the same labelled bounce path
ribbon parallelogram polyominoes. This can happen iff $T_1=T_2$ since $\phi$ is a bijection in its own right.
Since the labels of rows and columns are the same for the underlying tableaux, we therefore must have that the surplus number of cells
added to each is the same, i.e. $a=b$. This implies $M_1=M_2$, which means $\mydecpara$ is injective.
\ \\[1em]
Let $D \in \DecPara_{m,n}$. 
Let $B$ be the labelled bounce path ribbon parallelogram polyomino associated with $D$. That $B$ exists is straightforward: 
every parallelogram polyomino has a bounce path, and such a bounce path may be labelled in a way consistent with $D$.
Then we have $D=\lrib{B,\zeta}$ for some $B \in \DecRib_{m,n}$ and sequence $\zeta$ that supports $B$.
As $B \in \DecRib_{m,n}$, by Theorem~\ref{three:eight}
we must also have that there is a unique $T \in \EW_{m,n}$ such that $\phi(T) = B$.
This consideration shows that there must be a pair $(T,a) \in \MEW_{m,n}$ such that $\mydecpara(T,a) = D$, where $a_i = \eta_i(T) - \zeta_i - 1$, and so $\mydecpara$ is surjective.
\end{proof}

The bijection $\mydecpara$ of Definition~\ref{map:one} admits the following more self-contained presentation:

\begin{definition}\label{map:two}
Let $M=(T,a) \in \MEW_{m,n}$. 
If $v_i$ is a row of $M$, then let $\magic_i$ 
be the number of cornersupport 0s in that row plus the number of non-cornersupport 0s weakly to the left of the marked 0 in that row.
If $v_i$ is a column of $M$, then let $\magic_i$ be the number of cornersupport 1s in that column plus the number of non-cornersupport 1s weakly above the marked 1 in that column.
Define $D=(T',\ell) = \mydecpara(M)$ to be the decorated parallelogram polyomino that results from the following construction:
\begin{enumerate}
\item[(i)]
Let $\pi$ be the lexicographically smallest permutation such that the sequence $(\magic_{\pi(i)})_{i=1}^{m-1}$ is weakly increasing.
Let $\sigma $ be the lexicographically smallest permutation such that the sequence $(\magic_{m-1+\sigma(i)})_{i=1}^{n}$ is weakly increasing.
\item[(ii)] Let $\ell := (v_{\pi(1)},\ldots,v_{\pi(m-1)},v_{m-1+\sigma(i)},\ldots,v_{m-1+\sigma(n)}).$
\item[(iii)] If row $x$ of $D$ has label $v_i$, then the leftmost 1 in this row is in $T'_{x,\magic_{\pi(x)}}$.
\item[(iv)] If column $y$ of $D$ has label $v_i$, then the topmost 1 in this column is in $T'_{\magic_{m-1+\sigma(y)},y}$.
\item[(v)] Complete the interior of $D$ with 1s and insert 0s into any remaining empty cells.
\end{enumerate}
\end{definition}

\begin{example}
Let us consider the marked EW-tableau from Example~\ref{second:last}. We illustrate it here with non-cornersupport entries indicated by shaded cells.
\begin{center}
\begin{tikzpicture}
\def\minix{0.21}
\def\miniy{0.21}
\def\xx{0.435}
\def\bx{0}\def\by{0}
\def\initx{0.28}\def\inity{0.37}
\def\jnitx{0.31}\def\jnity{0.47}
\def\stepx{0.431}\def\stepy{0.431}
\newcommand\myrect[2]{\fill[black!30!white] (0.14+#1*\stepx,-0.14-#2*\stepy) rectangle ++(\stepx,-\stepy);}
\myrect{1}{0} \myrect{6}{0} \myrect{10}{0} \myrect{11}{0} \myrect{12}{0}
\myrect{0}{1} \myrect{2}{1} \myrect{3}{1} \myrect{4}{1} \myrect{1}{2}
\myrect{5}{2} \myrect{6}{2} \myrect{7}{2} \myrect{8}{2} \myrect{9}{2} \myrect{10}{2} \myrect{11}{2} \myrect{12}{2}
\myrect{0}{3} \myrect{5}{3} \myrect{7}{3} \myrect{8}{3} \myrect{9}{3}
\myrect{1}{4} \myrect{5}{4} \myrect{6}{4} \myrect{7}{4} \myrect{8}{4} \myrect{9}{4} \myrect{10}{4}
\myrect{11}{4} \myrect{12}{4} \myrect{0}{5}
\myrect{2}{5} \myrect{3}{5} \myrect{4}{5} \myrect{1}{6} \myrect{5}{6} \myrect{6}{6} \myrect{7}{6}
\myrect{8}{6} \myrect{9}{6} \myrect{10}{6} \myrect{11}{6} \myrect{12}{6}
\EWdiagramlabwithin{1\os1111\os111\os\os\os,\zs01\os100000000,10111101\os100\zs,\os011100\zs00000,101111\zs\os1\os000,\zs0\os1100000000,10111\os01110\zs0}{0,1,2,3,4,5,6}{7,8,9,10,11,12,13,14,15,16,17,18,19}{$M=(T,a)=$}
\end{tikzpicture}
\end{center}
In row $v_1$ there are nine cornersupport 0s and one non-cornersupport 0 weakly to the left of the marked 0, so $\magic_1=9+1=10$.
In row $v_2$ there are zero cornersupport 0s and five non-cornersupport 0s weakly to the left of the marked 0, so $\magic_2=0+5=5$.
In row $v_3$ there are five cornersupport 0s and two non-cornersupport 0s weakly to the left of the marked 0, so $\magic_3=5+2=7$.
In row $v_4$ there are zero cornersupport 0s and two non-cornersupport 0s weakly to the left of the marked 0, so $\magic_4=0+2=2$.
In row $v_5$ there are nine cornersupport 0s and one non-cornersupport 0 weakly to the left of the marked 0, so $\magic_5=9+1=10$.
In row $v_6$ there are zero cornersupport 0s and four non-cornersupport 0s weakly to the left of the marked 0, so $\magic_6=0+4=4$.

The lexicographically smallest permutation for which the sequence $(\magic_{\pi(1)},\ldots,\magic_{\pi(6)})$ is weakly increasing is
$\pi = (4,6,2,3,1,5)$.
We thus have the first segment of the labels $\ell$ of $D$ as $(\ell_1,\ldots,\ell_6) = (v_4,v_6,v_2,v_3,v_1,v_5)$ and $(\magic_{\pi(1)},\ldots,\magic_{\pi(6)}) = (2,4,5,7,10,10)$.
Also, the leftmost 1s in the rows of $D$ are at positions $T'_{12},T'_{24},T'_{35},T'_{47},T'_{5\,10},T'_{6\,10}$.

In column $v_7$ there are four cornersupport 1s and one non-cornersupport 1 weakly above the marked 1, so $\magic_7=4+1=5$.
Repeating this for the other columns we find $(\magic_7,\ldots,\magic_{19})=(5,1,7,6,7,4,1,3,2, 3  1, 1, 1)$.
The lexicographically smallest permutation for which the sequence $(\magic_{6+\sigma(1)},\ldots,\magic_{6+\sigma(13)})$ is weakly increasing
is $\sigma = (2,7,11,12,13,9,8,10,6,1,4,3,5)$.
The remaining labels are 
\begin{align*}
(\ell_7,\ldots,\ell_{19}) = &
(v_{6+2},v_{6+7},\ldots,v_{6+3},v_{6+5})\\ 
=& (v_8,v_{13},v_{17},v_{18},v_{19},v_{15},v_{14},v_{16},v_{12},v_{7},v_{10},v_{9},v_{11}).
\end{align*}
The highest 1s in each of the columns of $D$ are thus at positions:
\begin{align*}
&(T'_{\magic_{6+\sigma(1)},1} , T'_{\magic_{6+\sigma(2)},2},\ldots,T'_{\magic_{6+\sigma(13)},13} )\\
&= ( T'_{11},T'_{12},T'_{13},T'_{14},T'_{15},T'_{26},T'_{37},T'_{38},T'_{49},T'_{5\, 10},T'_{6\,11},T'_{7\,12},T'_{7\,13}).
\end{align*}
This gives the following diagram for `boundary 1s' of the labelled parallelogram polyomino:
\begin{center}
\begin{tikzpicture}
\EWdiagramlabwithin{11111~~~~~~~~,~1~~~1~~~~~~~,~~~1~~11~~~~~,~~~~1~~~1~~~~,~~~~~~1~~1~~~,~~~~~~~~~11~~,~~~~~~~~~1~11}{0,4,6,2,3,1,5}{8,13,17,18,19,15,14,16,12,7,10,9,11}{}
\end{tikzpicture}
\end{center}
Use these 1s that describe the boundary of the polyomino to complete the interior of the polyomino:
\begin{center}
\begin{tikzpicture}
\EWdiagramlabwithin{11111~~~~~~~~,~11111~~~~~~~,~~~11111~~~~~,~~~~11111~~~~,~~~~~~1111~~~,~~~~~~~~~11~~,~~~~~~~~~1111}{0,4,6,2,3,1,5}{8,13,17,18,19,15,14,16,12,7,10,9,11}{}
\end{tikzpicture}
\end{center}
Fill the remaining cells with 0s to yield the labelled parallelogram polyomino
\begin{center}
\begin{tikzpicture}
\EWdiagramlabwithin{1111100000000,0111110000000,0001111100000,0000111110000,0000001111000,0000000001100,0000000001111}{0,4,6,2,3,1,5}{8,13,17,18,19,15,14,16,12,7,10,9,11}{$\mydecpara(M)=$}
\end{tikzpicture}
\end{center}
\end{example}

\begin{remark}
The construction in Definition~\ref{map:two} is now helpful in understanding how Definition~\ref{myalg1} works.
We may consider an EW-tableau in $EW_{m,n}$ to be a marked EW-tableau wherein for every row (other than the first) the marked 0 in is the rightmost non-cornersupport 0 and for every column the marked 1 is the lowest non-cornersupport 1.
\end{remark}


\begin{thebibliography}{99}
\bibitem{aadhlb}
J.-C. Aval, M.~D'Adderio, M.~Dukes, A.~Hicks, and Y.~Le~Borgne.
\newblock Statistics on parallelogram polyominoes and a {\it q,t}-analogue of
  the {N}arayana numbers.
\newblock {\em Journal of Combinatorial Theory, Series A}, 123(1):271--286,
  2014.

\bibitem{aadlb}
J.-C. Aval, M.~D'Adderio, M.~Dukes, and Y.~Le~Borgne.
\newblock Two operators on sandpile configurations, the sandpile model on the
  complete bipartite graph, and a cyclic lemma.
\newblock {\em Advances in Applied Mathematics}, 73:59--98, 2016.

\bibitem{clb}
R.~Cori and Y.~Le~Borgne.
\newblock The sand-pile model and {T}utte polynomials.
\newblock {\em Advances in Applied Mathematics}, 30(1-2):44--52, 2003.

\bibitem{cp}
R.~Cori and D.~Poulalhon.
\newblock Enumeration of {\it (p,q)}-parking functions.
\newblock {\em Discrete Mathematics}, 256(3):609--623, 2002.

\bibitem{cr}
R.~Cori and D.~Rossin.
\newblock On the sandpile group of dual graphs.
\newblock {\em European Journal of Combinatorics}, 21(4):447--459, 2000.

\bibitem{Dhar}
D.~Dhar.
\newblock Theoretical studies of self-organized criticality.
\newblock {\em Physica A: Statistical Mechanics and its Applications},
  369(1):29--70, 2006.
  
\bibitem{dlb}
M.~Dukes and Y.~Le~Borgne.
\newblock Parallelogram polyominoes, the sandpile model on a complete bipartite
  graph, and a {\it q,t}-{N}arayana polynomial.
\newblock {\em Journal of Combinatorial Theory, Series A}, 120(4):816--842,
  2013.

\bibitem{ds}
M.~Dukes and T.~Selig.
\newblock Decomposing recurrent states of the {A}belian sandpile model.
\newblock {\em S{\'e}minaire Lotharingien de Combinatoire}, 77:B77g, 2018.

\bibitem{evw}
R.~Ehrenborg and S.~van Willigenburg.
\newblock Enumerative properties of {F}errers graphs.
\newblock {\em Discrete \& Computational Geometry}, 32(4):481--492, 2004.

\bibitem{Red}
F.~Redig.
\newblock Mathematical aspects of the abelian sandpile model.
\newblock In {\em Lecture Notes of Les Houches Summer School 2005, Mathematical
  Statistical Physics, Session LXXXIII}. Elsevier, 2006.

\bibitem{sss}
T.~Selig, J.~P. Smith, and E.~Steingr{\'\i}msson.
\newblock {EW}-tableaux, {L}e-tableaux, tree-like tableaux and the {A}belian sandpile model.
\newblock {\em The Electronic Journal of Combinatorics}, 25(3):3.14, 2018.

\bibitem{dsss}
M.~Dukes, T.~Selig, J.~P.~Smith, and E.~Steingr\'imsson.
\newblock {T}he {A}belian sandpile model on {F}errers graphs -- A classification of recurrent configurations.
\newblock {\em European Journal of Combinatorics} 81:221-241, 2019.

\bibitem{dasm}
M.~Dukes, T.~Selig, J.~P.~Smith, and E.~Steingr\'imsson.
{P}ermutation graphs and the {A}belian sandpile model, tiered trees and non-ambiguous binary trees.
\newblock {\em {E}lectronic {J}ournal of {C}ombinatorics} 26(3):P3.29, 2019.
\end{thebibliography}
\end{document}